\PassOptionsToPackage{prologue,dvipsnames}{xcolor}

\documentclass[12pt]{amsart}
\usepackage[headings]{fullpage}

\usepackage{texmaX,semmaX,semtkzX}
\usepackage{makecell}
\usepackage{fancyhdr}
\usepackage{adjustbox}

\usepackage{ocgx2}

\usepackage{amssymb}
\usepackage{amsmath}
\usepackage{amsthm}
\usepackage{multicol}
\usepackage{mathrsfs}
\usepackage{graphicx}
\usepackage{url}
\usepackage[all]{xy}
\usepackage{tabularx}
\usepackage{float}
\usepackage{psfrag}
\usepackage{tikz}
\usetikzlibrary{positioning}

\usepackage[utf8]{inputenc}
\usepackage[active]{srcltx}
\usepackage{fullpage}
\usepackage{amsfonts}
\usepackage{comment}
\usepackage{enumitem}
\usepackage[mathscr]{euscript}
\usepackage{epsfig}
\usepackage{latexsym}
\usepackage[all]{xy}
\usepackage{amssymb}
\usepackage{amsthm}
\usepackage{amsmath}
\usepackage{float}
\usepackage{amsxtra}
\usepackage[dvipsnames]{xcolor}

\usepackage{amssymb}
\usepackage{pifont}

\def\s{\mathfrak s}
\def\id#1{{\mathfrak{#1}}}      

\def \ZZ{\mathbb{Z}}
\def \Z{\mathcal{Z}}

\def\<#1>{{\left\langle{#1}\right\rangle}}

\DeclareMathOperator{\norm}{{\mathcal N}}

\theoremstyle{plain}

\usepackage[OT2,OT1]{fontenc}
\usepackage{pbox}
\usepackage{tkz-graph}
\usetikzlibrary{decorations,fit,positioning,calc}

\makeatletter
\newcommand{\verbatimfont}[1]{\renewcommand{\verbatim@font}{\ttfamily#1}}
\makeatother

\def\ZZ{\mathbb Z}

\def\Z{{\mathbb Z}}             
\def\Q{{\mathbb Q}}             

\newcommand{\HGM}{\text{HGM}}

\newcommand\TODO[1]{\textcolor{red}{#1}}

\usepackage[utf8]{inputenc}
\usepackage[pagewise]{lineno} 

\begin{document}

\title{On the conductor of a family of \text{Frey hyperelliptic curves}}

\author{Pedro-Jos\'e Cazorla Garc\'ia}
	
\address{Departamento de Matem\'atica Aplicada, ICAI, Universidad Pontificia Comillas, Madrid, 28015, Spain}  \email{pjcazorla@comillas.edu}

\author{Lucas Villagra Torcomian}
	
\address{Department of Mathematics, Simon Fraser University, Burnaby, BC V5A 1S6, Canada.}  \email{lvillagr@sfu.ca}


\begin {abstract}
In the breakthrough article \cite{Darmon}, Darmon presented a program to study Generalized Fermat Equations (GFE) via abelian varieties of $\GL_2$-type over totally real fields.
So far, only  Jacobians of some Frey hyperelliptic curves have been used with that purpose. In the present article, we  show how most of the known Frey hyperelliptic curves are particular instances of a more general biparametric family of hyperelliptic curves $C(z,s)$. Then, we apply the cluster picture methodology to compute the conductor of $C(z,s)$ at all odd places, for some general parameters $z,s \in \Z$, not related a priori to any GFE.

As a Diophantine application, we specialize $C(z,s)$ in
some particular values $z_0$ and $s_0$, and we find the conductor exponent at odd places of the natural Frey hyperelliptic curves attached to $Ax^p+By^p=Cz^r$ and $Ax^r+By^r=Cz^p$, generalizing the results on \cite{CelineClusters} and opening the door for future research in GFE with coefficients. Moreover, we show how a new Frey hyperelliptic curve for $x^2+By^r=Cz^p$ can be constructed in this way, giving new results on the conductor exponents for this equation. 

Finally, following the recent approach in \cite{GolfieriPacetti}, we consider the Frey representations attached to a general signature $(q,r,p)$ via hypergeometric motives and, using $C(z,s)$, we compute the wild part of the conductor exponent at primes above $q$ and $r$ of the residual representation modulo a prime above $p$.
\end {abstract}

\maketitle

\section*{Introduction}

\subsection*{Historical background}

Dating back to Ancient Greece, Diophantine equations are among the most classical areas of mathematics, and include many of its most famous problems, such as Fermat's Last Theorem, originally stated by Fermat in the 17th century. This result affirms that the only integer solutions $(a,b,c)$ to the equation
 \begin{equation}
 \label{eqn:flt}
 x^n+y^n=z^n, \quad n\ge3,
 \end{equation}
are the trivial ones, i.e.\ those satisfying $xyz=0$. After Wiles' proof of Fermat's assertion \cite{Wiles}, many researchers have considered generalizations of \eqref{eqn:flt}, which are broadly known as \emph{Generalized Fermat Equations} (or GFE for short). These are Diophantine equations of the form
\begin{align}\label{eq:GFE}
    Ax^q+By^r=Cz^p,
\end{align}
 where $A$, $B$ and $C$ are fixed non-zero  integers. The triple of exponents $(q,r,p)$, which are allowed to be variables, is usually called the \textit{signature} of the equation.  
 
 Matiyasevich \cite{Matiyasevich} answered Hilbert's tenth problem by showing that there cannot be a general algorithm for resolving an arbitrary Diophantine equation, which, in particular, makes the determination of all solutions of \eqref{eq:GFE} an extremely difficult problem.
 
However, the strategy pioneered by Wiles for the study of \eqref{eqn:flt}, which is now referred to as \textit{the modular method}, applies to a wide variety of GFE, and can be used to show that some instances of \eqref{eq:GFE} do not have solutions.

For the convenience of the reader, we give a high-level overview of the modular method (see \cite{Siksek,SiksekKhawaja} for a more detailed explanation). Without loss of generality, let us assume that $p$ is a variable prime in \eqref{eq:GFE}. The main steps of the modular method are the following:
\begin{enumerate}
    \item Suppose for contradiction that \eqref{eq:GFE} has a non-trivial solution $(a,b,c)$\footnote{To be more precise, we want to consider non-trivial \textit{primitive} solutions, i.e.\ such that  $\gcd(a,b,c)=1$.}. Then, a Galois representation $\rho$ is attached to $(a,b,c)$. This representation should have bounded conductor and its image should lie in $\GL_2(\F_p)$.
    
    \item The representation $\rho$ is shown to be modular. In other words, it arises from a modular form of weight $2$ and level equal to the Serre level $N$  of $\rho$ (see \cite{SerreConjecture}). The key feature of the modular method is that $\rho$ can be constructed so that $N$ is independent of $(a,b,c)$.

    \item Finally, all representations coming from modular forms of weight $2$ and level $N$ are considered. If it can be shown that none of them are isomorphic to $\rho$, a contradiction is reached and \eqref{eq:GFE} is shown to not have non-trivial solutions.
\end{enumerate}

\subsection*{State of the art and methodology} 

In order to construct the representation $\rho$ in step (1), one is usually able to attach a geometric object to the solution. Classically, this object has been an elliptic curve, which receives the name of  \emph{Frey elliptic curve}. 
However, so far, it is only known how to associate a Frey elliptic curve for a few families of values $(q,r,p)$ (see e.g.\ \cite[Section 4]{DarmonGranville} for a discussion about this).

To overcome this difficulty, and in order to study new signatures, Darmon \cite{Darmon} proposes to replace elliptic curves by higher dimensional abelian varieties, which become of $\GL_2$-type over a certain totally real number field. Since Darmon's celebrated paper, Jacobians of \emph{Frey hyperelliptic curves} have begun to be considered as a source of the Galois representation $\rho$ (see for example \cite{BCDF, BCDF2, ChenKoutsianas}). The interested readers can look at the survey article \cite{ChenKoutsianasSurvey} for a more detailed explanation of the state of the art in the subject.

In his article, Darmon showed how the varieties appearing in the program share some Galois structures.
Moreover, he explained in a clear diagram how to propagate modularity between them once suitable modularity
lifting theorems become available (see \cite[p.\ 433]{Darmon}). Due to recent breakthroughs in this direction, proving modularity is no longer the main difficulty of the method. On the other hand, step (2) requires the computation of the Serre level of the Galois representation, and so it is necessary to determine the Frey curve's conductor. When working with elliptic curves, this is straightforward thanks to Tate's algorithm \cite{Tate}. However, there is no routine to compute the conductor when working with hyperelliptic curves of genus $g\ge2$, and so its calculation becomes a substantially harder problem. 

To overcome the lack of an algorithm, the pioneer article by Dokchitser, Dokchitser, Maistret and Morgan \cite{DDMM} introduced the framework of \emph{cluster pictures}. Broadly speaking, the differences of the roots of the defining polynomial of the curve are studied in order to construct a combinatorial object, from which the conductor of the curve at odd places can be read. 
We invite readers interested in the topic of cluster pictures to also read the expository article \cite{BBB}.

In the recent paper \cite{CelineClusters}, the aforementioned methodology of cluster pictures has been used for some well-known Frey hyperelliptic curves from the literature. Specifically, the authors study Kraus's curve $C_{r,r}$ \cite{Kraus}, as well as Darmon's curves $C^{\pm}_r$ \cite{Darmon}, used, respectively, for signatures $(r,r,p)$ and $(p,p,r)$.

\subsection*{Our contribution} 

Currently, as we previously discussed, the main challenge of the modular method lies in the computation of the conductor, and this is why we will focus on this problem in the present article. 
Specifically, inspired by the work on \cite{CelineClusters}, we compute the conductor exponents at odd primes of the following biparametric family of hyperelliptic curves:
\begin{align}
    \label{eqn:C}
      C(z,s)\colon y^2 =  (-z)^\frac{r-1}{2}xh\left(2-\frac{x^2}{z}\right)+s,
\end{align}
 where $s, z\in \ZZ$ are any integers satisfying \ref{item1} and \ref{item2} (not necessarily linked to any GFE), {$r\ge5$ is a prime number} and $h(x)$ is the minimal polynomial of $\zeta_r+\zeta_r^{-1}$, with $\zeta_r$ a primitive $r$-th root of the unity. Our motivation to study this family is twofold. Firstly, we shall see in Section \ref{sec:applications} that certain values of $z$ and $s$ give rise to Frey hyperelliptic curves for GFE of signatures $(p,p,r)$, $(r,r,p)$ and $(2,r,p)$. By working with $C(z,s)$ before specializing the relevant values of $z$ and $s$, we can adopt a uniform treatment for all three cases.

We remark that the computation of the conductor for these three signatures represents a significant advancement in this direction, since, for $r\ge5$, the conductor has only been computed for $C_{r,r}$ and $C^\pm_r$, i.e.\  when $A=B=C=1$, while our calculations work in full generality. Also, to the best of our knowledge, there are no results in the literature on the conductor of the Frey hyperelliptic curve for the signature $(2,r,p)$. While this article was under the submission process, we became aware of the concurrent work of Azon \cite{Martinpreprint}, where the author also provides improvements for the signatures $(r,r,p)$ and $(p,p,r)$ involving coefficients. Azon's paper includes different computations, such as the inertial type of Galois representations and the conductor exponent at primes above $2$. These aspects are beyond the scope of this article. While the computations of the conductor at odd primes match for the signatures $(r,r,p)$ and $(p,p,r)$ (see Remark \ref{rem:martin}), the remaining Diophantine applications in Section \ref{sec:applications} are not addressed in \cite{Martinpreprint}.

Our second application to $C(z,s)$ is more ambitious. In the recent promising paper \cite{GolfieriPacetti} the authors show how to construct a Galois representation (satisfying the standard desired conditions) for an arbitrary signature $(q,r,p)$, through hypergeometric motives. Moreover, it is known how  to calculate the conductor exponent of such a representation for primes not dividing $qr$ (see the second theorem in page 6 of \cite{GolfieriPacetti}). In Section \ref{sec:applications}, we relate this representation with the one attached to the Jacobian of $C(z,s)$, allowing us to compute the wild conductor exponent at the remaining cases, i.e.\ at primes above $q$ and $r$. The combination of the results in \cite{GolfieriPacetti} with our findings opens the door to the application of the modular method for any signature $(q,r,p)$.


\vspace{5pt}

Having established the scope of the article, we now present the  main result, which allows us to compute the odd conductor exponents of $C = C(z,s)$, both over $\Q$ and over $K = \Q(\zeta_r + \zeta_r^{-1})$, where $\zeta_r$  is a  primitive $r$-th root of unity. 

\begin{thm*} 
Let $M \in \{\Q, K\}$ and let $C/M$ be the curve defined in \eqref{eqn:C} where $z$ and $s$  satisfy $\ref{item1}$ and $\ref{item2}$. Then the conductor exponent of $C$ for all odd primes of bad reduction is given in Table~\ref{table:conductorQ} for $M = \Q$ and in Table~\ref{table:conductorK} for $M = K$. 
\end{thm*}

\subsection*{Organization of the paper}
A brief introduction to the basic notions about hyperelliptic curves and cluster pictures is given in Section \ref{Sec:preliminaries}. In a short Section \ref{Sec:roots}, we explicitly determine the roots of the polynomial $F(x)$ defining the curve (\ref{eqn:C}), along with some of their basic properties. In the following two sections we prepare the scenario for proving the main theorem: in Section \ref{sec:clusterpictureC} we find the cluster picture attached to (\ref{eqn:C}), both over $\Q$ and over $K$, while in Section \ref{Sec:ramification} we study in more detail the polynomial $F(x)$, by determining its splitting field and the associated ramification index, as well as giving criteria for the irreducibility of $F(x)$. In Section \ref{Sec:conductor}, we prove the main results of the article by explicitly computing the conductor exponent at odd places for  (\ref{eqn:C}). Finally, in Section \ref{sec:applications} we show some of the applications to GFE, by specializing (\ref{eqn:C}) to particular values.

\subsection*{Acknowledgements} We thank Tim Dokchitser for sharing with us his \verb|Magma|'s code to compute cluster pictures, as well as C\'eline Maistret for allowing us to use the package of \cite{BBB} to draw the cluster pictures. We also thank Imin Chen and Franco Golfieri for valuable comments, and
Ariel Pacetti for helpful discussions. We would like to thank the anonymous referee for valuable comments that helped to improve the quality of this work. 

\subsection*{Funding} The first author is supported by Universidad Pontificia Comillas. The second author is supported by a PIMS-Simons postdoctoral fellowship.  

\subsection*{Data Availability} Data sharing not applicable to this article as no datasets were generated or analyzed during the current study. The \texttt{Magma} programs supporting computations in this paper can be found in \cite{github}.

\subsection*{Conflicts of interest} The authors have no conflicts of interest to disclose.

\section{Preliminaries}
\label{Sec:preliminaries}

In this section, we will present the tools that we shall use throughout this article, along with the main results that we will need. For simplicity, this section is broken into several subsections.

\subsection{Hyperelliptic curves}

In this section, we define some concepts about hyperelliptic curves that will be relevant for us. The exposition in this section is by no means exhaustive and we invite the interested reader to consult \cite{Liu1,Liu2,DDMM} for an excellent review of the topic. For now, let us introduce two basic definitions.

\begin{definition}
    \label{def:hyperellipticcurve}
    Let $g \ge 2$ be an integer, let $K$ be a field of characteristic not equal to $2$, and let $f(x) \in K[x]$ be a separable polynomial of degree $2g+1$ or $2g+2$. Suppose that the discriminant of $f$ is not equal to $0$. Then, the curve with affine equation 
    \begin{equation}
    \label{eqn:weierstraass}
    C: y^2 = f(x)
    \end{equation}
    is called a \textit{hyperelliptic curve of genus} $g$.
\end{definition}

\begin{definition}
\label{def:discriminantcurve}
    Let $C$ be a hyperelliptic curve of genus $g$ with equation as in \eqref{eqn:weierstraass}, and assume that $f(x)$ is monic and has degree $2g+1$. Let $\Delta_{f(x)}$ be  the discriminant of the polynomial $f(x)$. Then the quantity
    \[\Delta_C = 2^{4g}\Delta_{f(x)},\]
    is called the \textit{discriminant} of $C$.
\end{definition}



The main object of study of the present article is the conductor associated to a curve $C$. As shown by the following definition, the conductor is a product of the primes dividing $\Delta_C$, raised to a certain exponent.

\begin{definition}
    \label{def:conductor}
    Let $K$ be a number field with ring of integers $\cO_K$ and let $C/K$ be a hyperelliptic curve as in \eqref{eqn:weierstraass}. The \textit{conductor} of $C$ is the ideal of $\cO_K$ given by
    \[\mathfrak{n}(C) = \prod_{\id{q} \mid \Delta_C} \id{q}^{\mathfrak{n}(C/K_{\id{q}})},\]
    where $K_\id{q}$ denotes the completion of $K$ at $\id{q}$ and the integer ${\mathfrak{n}(C/K_{\id{q}})}$ is called the \textit{conductor exponent} of $C$ at $\id{q}$, which can be defined as in \cite[Section 2.1]{Serre}. The primes $\id{q}$ for which ${\mathfrak{n}(C/K_{\id{q}})} > 0$ are called \emph{primes of bad reduction for $C$}.
\end{definition}

\begin{remark}
   To be fully precise, we would need to define    the conductor in terms of the Tate module and Galois actions (see for example \cite[Section 4]{Ulmer}). In order to simplify the exposition, we have preferred to stay away from these technicalities and use the previous ``working definition''.
\end{remark}

We remark that the computation of the conductor exponents ${\mathfrak{n}(C/K_{\id{q}})}$ is highly non-trivial for curves of genus $g \ge 2$ and requires the use of many of the techniques that we present in this article. The conductor exponent is composed of two parts, which depend on the splitting field of the defining polynomial of the curve, as shown in Lemma \ref{lemma:twoparts}. Before that, we introduce the following definition.

\begin{definition}
    Let $p$ be a prime number and let $L$ and $K$ be local fields with  $\Q_p \subseteq K \subseteq L$. Let $e_{L/K}$ denote the ramification index of the extension $L/K$. We say that $L/K$ is a \textit{wild extension} if $p \mid e_{L/K}$, and a \textit{tame extension} otherwise. 
\end{definition}

\begin{lemma}
    \label{lemma:twoparts}
    Let $p$ be a prime number, $K$ a local field with $\Q_p \subseteq K$ and let $C/K$ be a hyperelliptic curve of genus $g$ with defining polynomial $f(x)$. Then, the conductor exponent $\mathfrak{n}(C/K)$ decomposes as
    \[\mathfrak{n}(C/K) = \mathfrak{n}_{wild}(C/K) + \mathfrak{n}_{tame}(C/K),\]
    where $\mathfrak{n}_{wild}(C/K)$ and $\mathfrak{n}_{tame}(C/K)$ are called, respectively, the wild and the tame part of the conductor. In addition, if $L$ is the splitting field of $f(x)$, we have that $\mathfrak{n}_{wild}(C/K) = 0$ if the extension $L/K$ is tame.
\end{lemma}

\begin{proof}
    This is \cite[Definition 2.2 and Proposition 2.11]{CelineClusters}.
\end{proof}

As we shall see in the following section, the behaviour of the wild and tame part of the conductor is quite different, and so is the manner of computing them (see Theorem \ref{thm:formulascluster}).

\subsection{Clusters}
\label{Sec:introclusters}
Let $p \neq 2$ be a prime number, and let $K$ be a local field with $\Q_p \subseteq K$, absolute Galois group $G_K$, ring of integers $\cO_K$, normalized valuation $v$ and uniformizer $\pi$.

Let  $C/K$ be a hyperelliptic curve of genus $g$ given by a Weierstrass equation as in \eqref{eqn:weierstraass}.
We denote by $\mathcal{R}$ the set of roots of $f(x)$ over $\bar{K}$, which is going to be a set of size $2g+1$ or $2g+2$.

\begin{definition} A \textit{cluster} $\s$ is a non-empty subset of $\mathcal{R}$ of the form $\s=D\cap \mathcal{R}$ for some disc $D=\{x\in\bar{K} : v(x-z)\ge d\}$, where $z\in \bar{K}$ and $d\in\Q$.
\end{definition}

It is well known that given two clusters $\s_1$ and $\s_2$, then either they are disjoint or one is contained in the other (see \cite[Section 1.5]{DDMM}). This motives the following definition.

\begin{definition} If $\s,\s'$ are two clusters such that $\s'\subsetneq \s$ is a maximal subcluster, we refer to $\s'$ as a \textit{child} of $\s$. Consistently, $\s$ will be the \textit{parent} of $\s'$, and we will denote it by $\s = P(\s')$. 
\end{definition}

Let $L$ be the splitting field of $f(x)$ and let $I_K$ be the inertia group of $\Gal(L/K)$. If $A$ is any set where $I_K$ acts on, we let $A/I_K$ be the set of orbits of $A$ under the action of $I_K$. Furthermore, if $\s$ is a cluster, we denote by $I_\s$ the stabiliser of $\s$ under $I_K$. Moreover, if there exists a unique fixed child $\s'$ under the action of $I_\s$, then we call $\s'$ an \textit{orphan} (of $\s$).

\vspace{5pt}

We denote by $|\s|$ the number of elements of $\mathcal{R}$ contained in $\s$. We will abuse the language by saying that $\s$ is \textit{odd} (resp.\ \textit{even}) if $|\s|$ is odd (resp.\ even).  Also, if $|\s|>1$ we  say that $\s$ is \textit{non-trivial}, and in such a case we can define its \textit{depth} $d_\s$ to be 
\[d_\s=\min \{v(r-r') : r,r'\in \s\}.\]
Moreover, if $\s$ is a \textit{proper} cluster, i.e.\ $\s\neq \mathcal{R}$, then we define its \textit{relative depth} to be 
\[\delta_\s=d_\s-d_{P(s)}.\]

\begin{notation}
In a cluster picture, we denote by \smash{\raise4pt\hbox{\clusterpicture\Root[A]{1}{first}{r1};\endclusterpicture}} each element of $\mathcal{R}$, while clusters are represented by an oval, such as:

        \[ \clusterpicture            
	\Root[A] {2} {first} {r1};
	\Root[A] {3} {r1} {r2};
	\ClusterLDName c1[][2][] = (r1)(r2);
        \Root[A] {6} {r2} {r3};
        \Root[A] {3} {r3} {r4};
        \ClusterLDName c2[][2][] = (r3)(r4);
        \Root[A] {6} {r4} {r5};
	\ClusterLDName c[][1][] = (c1)(c2)(r4)(r5);
	\endclusterpicture \]

The subscript on the largest cluster $\mathcal{R}$ is its depth, while subscripts on the other clusters are their respective relative depths. For instance, in the above picture we have 5 roots: the largest cluster $\mathcal{R}$ has depth 1 and the two twins (see Definition \ref{def:twin}) have depth 3.

\end{notation}

\begin{definition}
\label{def:twin}
A cluster $\s$ is a \textit{twin} if $|\s|=2$. An even cluster whose children are all even is called \textit{übereven}.
\end{definition}

\begin{definition} For two clusters (or roots) $\s_1$ and $\s_2$ we write $\s_1 \wedge \s_2$ for the smallest cluster containing both of them.
\end{definition}

Given a cluster $\s$, we denote by $\tilde{\s}$ the set of odd children of $\s$ and consider the following two quantities:
\[\tilde{\lambda}_\s=\frac{1}{2}\left(v(c)+|\tilde{\s}|d_\s + \sum_{\gamma\notin \s} d_{\{\gamma\}\wedge\s}\right) \ \text{ and } \ \nu_\s=v(c) +|\s|d_\s+\sum_{\gamma\notin\s}d_{\{\gamma\}\wedge\s},\]
where $c$ is the leading coefficient of $f(x)$. In addition, if $\s$ is proper and $a\in\Q$, we define $\xi_\s(a)=\max\{-v_2([I_K:I_\s]a),0\}$, where $v_2$ is the $2$-adic valuation of $\Q$. 
\vspace{4pt}

In order to state the main result that we will constantly use to compute the conductor of a hyperelliptic curve $C/K$, we first need to define the following sets associated to a given cluster picture.
\begin{align*}
    & U = \{\s : \s\neq\mathcal{R} \text{ is odd and } \xi_{P(\s)}(\tilde{\lambda}_{P(\s)})\le \xi_{P(\s)}(d_{P(\s)})\},\\
    & V = \{\s : \s\neq\mathcal{R} \text{ is non-übereven and } \xi_\s(\tilde{\lambda}_\s)=0\}.
\end{align*}

\begin{thm} 
\label{thm:formulascluster}
Let $C/K$ be a hyperelliptic curve given as in $(\ref{eqn:weierstraass})$, where $f(x)$ is monic. Then,
\begin{align*}
    &\mathfrak{n}_{\text{tame}}(C/K)=2g-\#(U/I_K)+\#(V/I_K).
\end{align*}
Moreover, if $f(x)$ is irreducible,
\[\mathfrak{n}_{\text{wild}}(C/K)=v_r(\Delta({K(\gamma_0)/K})) - \text{deg}(f) + f_{K(\gamma_0)/K},\]
where $\gamma_0\in\bar{K}$ is a root of $f(x)$, and $\Delta_{K(\gamma_0)/K}$ and $f_{K(\gamma_0)/K}$ are the discriminant and the residue degree of $K(\gamma_0)/K$, respectively.
\end{thm}
\begin{proof} This is a particular case of \cite[Theorem 12.3]{BBB}. 
\end{proof}

\subsection{Setup}

Let $r\ge5$ be a prime integer. Throughout this paper $\zeta_r$ will denote a fixed primitive $r$-th root of unity. For all $1\le j \le r-1$, we let $\omega_j=\zeta_r^j+\zeta_r^{-j}$ and $K=\Q(\omega_1)$, the maximal totally real subextension of the cyclotomic field $\Q(\zeta_r)$. Let 
\begin{equation}
    \label{eqn:h}
h(x)=\prod_{j=1}^{\frac{r-1}{2}}(x-\omega_j)
\end{equation}
be the minimal polynomial of $\omega_1$ and let
\begin{align*}
      C(z,s)\colon y^2 =  (-z)^\frac{r-1}{2}xh\left(2-\frac{x^2}{z}\right)+s
\end{align*}
be the curve defined in \eqref{eqn:C}, where $s$ and $z$ are integers such that if $q$ is a prime dividing both $s$ and $z$, then
\begin{enumerate}[label=(\roman*)]
\item  \label{item1} $v_q(z^{r})>v_q(s^2)$, 
\item  \label{item2} $r\nmid v_q(s)$.
\end{enumerate}
For all of the applications of arithmetic interest that we will consider in Section \ref{sec:applications}, these two conditions will be fulfilled. However, we believe that the approach that we take in this paper could be adapted even if conditions \ref{item1} and \ref{item2} were not satisfied, but we would need to deal with many more cases.

Some explicit examples of $C(z,s)$ are listed below.
\begin{align*}
    r=5:& \quad y^2 = x^5 - 5zx^3 + 5z^2x +s,  \\
    r=7:& \quad y^2 = x^7 -7zx^5 + 14z^2x^3 -7z^3x +s,\\
    r=11:& \quad y^2 = x^{11}-11zx^9+44z^2x^7-77z^3x^5+55z^4x^3-11z^5x+s.
\end{align*}

\begin{remark}
    \label{rmk:genus}
    From the expression of $C(z,s)$, it directly follows that the right-hand side is a polynomial in $x$ of degree $r$. Consequently, Definition \ref{def:hyperellipticcurve} gives that $C(z,s)$ is a hyperelliptic curve of genus $(r-1)/2$.
\end{remark}

\subsection{Some technical results}

Throughout the paper we will repeatedly use the fact that for any $1 \le k \le r-1$, $v_q(1-\zeta_r^k) = \delta_{qr}/(r-1)$, where $\delta_{ij}$ denotes the {Kronecker} delta.

We finish this section by presenting two useful lemmas, that we will need throughout the article. Let $q$ be an odd prime number, $M$ a finite extension of $\Q_q$ with valuation $v$ and we let $\cO_M$ denote the ring of integers of $M$.

\begin{lemma}[Lemma 2.15 \cite{CelineClusters}]\label{lemma:lematecnico}
 	
 	Let $\alpha, \beta \in \mathcal{O}_M$ be such that $v(\alpha^r + \beta^r) > 0$ and $v(\alpha)v(\beta) = 0$.
 	
 	\begin{enumerate}
 		\item If $q \neq r$, there exists some $0 \leq j_0 \leq r - 1$ such that $v(\alpha + \zeta_r^{j_0} \beta) = v(\alpha^r + \beta^r) > 0$ and $v(\alpha + \zeta_r^k \beta) = 0$ for all $0 \leq k \leq r - 1$, $k \neq j_0$.
 		\item If $q = r$, then $v(\alpha + \zeta_r^j \beta) > 0$ for any $0 \leq j \leq r - 1$. Moreover, if $v(\alpha^r + \beta^r) > rv(1 - \zeta_r)$ then there exists some $0 \le j_0 \le r-1$ such that $v(\alpha + \zeta_r^k \beta) = v(1 - \zeta_r)$ for any $0 \leq k \leq r - 1$, $k \neq j_0$.
 	\end{enumerate}
 	Furthermore, if $\zeta_r \notin M$, then we have $j_0 = 0$ in cases (1) and (2) respectively.
 \end{lemma}

We shall need the following extension of the previous lemma, which covers some of the cases not considered therein.

\begin{lemma}
    \label{lemma:lematecnico2}
    Let $\alpha$ and $\beta$ be as in Lemma $\ref{lemma:lematecnico}$ and suppose that $q=r$. Then, either
    \begin{enumerate}
        \item There exists $j_0$ with $0 \le j_0 \le r-1$ with $v(\alpha+\zeta_r^{j_0}\beta) > v(1-\zeta_r)$ and, in this case, $v(\alpha + \zeta_r^k \beta) = v(1 - \zeta_r)$ for any $0 \leq k \leq r - 1$, $k \neq j_0$, or 
        \item we have that $v(\alpha + \zeta_r^k \beta) = v(\alpha + \zeta_r^{k'} \beta)$ for any $0 \le k, k' \le r-1$.
    \end{enumerate}
    In the first case, if $\zeta_r \not \in M$, $j_0 = 0$.
\end{lemma}

\begin{proof}
    The assumptions $v(\alpha^r+\beta^r)>0$ and $v(\alpha)v(\beta)=0$ imply $v(\alpha)=v(\beta)=0$. In particular, for all $0<j\le r-1$, we have $v(\zeta_r^{j_0}(1-\zeta_r^{j})\beta)=v(1-\zeta_r)$.
    
    Suppose that the condition in (1) is satisfied, i.e.\  $v(\alpha + \zeta_r^{j_0}\beta) > v(1-\zeta_r)$ for some $0 \le j_0 \le r-1$. Then, for all $0 \le k \le r-1$ with $k \neq j_0$, we have that 
    \begin{align} \label{eq:case1}
    v(\alpha + \zeta_r^k\beta) &= v((\alpha + \zeta_r^{j_0}\beta) - (\zeta_r^{j_0}(1-\zeta_r^{k-j_0})\beta)) \nonumber \\
 & = \min\{v(\alpha+\zeta_r^{j_0}\beta),  v(1-\zeta_r^{k-j_0})\}   \\
   &= v(1-\zeta_r).   \nonumber 
    \end{align}

    Now, let us suppose that (1) does not hold (i.e.\ $v(\alpha+\zeta_r^k\beta) \le v(1-\zeta_r)$ for all $0 \le k \le r-1$)  and let us proceed to proving that condition (2) holds. 
    
    Firstly, let us assume that $v(\alpha+\zeta_r^{k}\beta) = v(1-\zeta_r)$
    for some $0 \le k \le r-1$.   Then, for $0\le k' \le r-1$ it follows that 
    \[v(\alpha + \zeta_r^{k'}\beta) = v(\zeta_r^{k'}(1-\zeta_r^{k-k'})\beta + (\alpha + \zeta_r^k\beta)) \ge \min \{v(1-\zeta_r), v(\alpha + \zeta_r^k\beta)\}= v(1-\zeta_r).
    \]
    However, (1) does not hold and consequently
    $v(\alpha+\zeta_r^{k'}\beta) = v(1-\zeta_r) = v(\alpha+\zeta_r^k\beta)$.
Finally, if   $v(\alpha+\zeta_r^{k}\beta) < v(1-\zeta_r)$
    for all $0 \le k \le r-1$,  by an identical argument it follows that 
    $v(\alpha + \zeta_r^{k'}\beta) = v(\alpha + \zeta_r^k\beta)$
    for all $0 \le k, k' \le r-1$. 

    \vspace{5pt}

    Finally, to prove the last statement, assume that the first case holds and that $\zeta_r\notin M$. Assume by contradiction that $j_0\neq0$. Then, there exists $\sigma\in\Gal(M(\zeta_r)/M)$ not fixing $\zeta_r^{j_0}$, i.e.\ $\sigma(\zeta_r^{j_0})=\zeta_r^{\ell_0}$ for some $1\le \ell_0\le r-1$, $\ell_0\neq j_0$. Hence
    \[v(\alpha+\zeta_r^{\ell_0}\beta)=v(\sigma(\alpha+\zeta_r^{j_0}\beta))=v(\alpha+\zeta_r^{j_0}\beta)>0,\]
    which contradicts (\ref{eq:case1}).
\end{proof}

\section{\texorpdfstring{Computing the roots}{Computing the roots}}
\label{Sec:roots}

The main aim of this article is to compute the conductor of $C(z,s)$ at all odd places, both over $\Q$ and over $K$. For this, we first need its cluster picture. To compute it, we will work mainly over local fields, and so we shall introduce some notation which we will keep throughout the article.

\begin{notation}
\label{notation:valuations}
Let $q \ge 3$ be an odd prime. Then, we will denote by $v_q(\cdot)$ the standard $q$-adic valuation of $\Q_q$. When considering a field extension $K/\Q_q$ with ramification degree $e_{K/\Q_q}$, we will also write $v_q(\cdot)$ to denote the unique extension of the valuation of $\Q_q$ to $K$ satisfying that $v_q(q) = 1$. Finally, if $\id{q}$ is the unique prime ideal of the ring of integers of $K$ above $q$, $v_{\id{q}}(\cdot)$ will denote the valuation of $K$ normalized so that $v_\id{q}(q) = e_{K/\Q_q}$. Note that, by our choice of notation, $v_{\id{q}}(x) = e_{K/\Q_q}v_{q}(x)$ for all $x \in K$. 
\end{notation}

Given $z$ and $s$ two fixed integers,  we set
\begin{equation}
    \label{eqn:F}
    F(x)=(-z)^\frac{r-1}{2}xh\left(2-\frac{x^2}{z}\right)+s
\end{equation}
to be the defining polynomial of $C(z,s)$. Let $q\ge3$ be a prime number and $\bar{\Q}_q$ be a fixed algebraic clousure of $\Q_q$. Set $\mathcal{R}=\{\gamma_0,\cdots,\gamma_{r-1}\}\subseteq \bar{\Q}_q$ to be the set of roots of $F(x)\in\Q_q[x]$. In this section, we shall give an explicit expression of the roots in $\bar{\Q}_q$, as shown in Proposition~\ref{prop:roots}. Some of the results that follow are
based on the arguments in \cite{CelineClusters}. 

\begin{prop}\label{prop:roots}
 For all $0\le j\le r-1$ we have that $\gamma_j=\zeta_r^j\alpha_0+\zeta_r^{-j}\beta_0$,
  where
 	\begin{equation}
        \label{eqn:alpha0beta0}
 	\alpha_0=-\sqrt[r]{\frac{s+\sqrt{s^2-4z^r}}{2}}\in\bar{\Q}_q, \quad \quad \beta_0=-\sqrt[r]{\frac{s-\sqrt{s^2-4z^r}}{2}}\in\bar{\Q}_q.
 	\end{equation}
 \end{prop}

\begin{proof}
	For each $0 \le j \le r-1$, we let $\alpha_j=\zeta_r^j\alpha_0$  and $\beta_j=\zeta_r^{-j}\beta_0$,
    so that $\gamma_j=\alpha_j+\beta_j$. 
    We note that 
\begin{equation}
    \label{eqn:rootproduct}
    \alpha_j\beta_j = \alpha_0\beta_0 =z \quad \text{ and } \quad  \alpha_j^r+\beta_j^r=\alpha_0^r+\beta_0^r = -s.
\end{equation}

	Then, it follows that
	\begin{align*}
	F(\gamma_j)&=(-z)^\frac{r-1}{2}(\alpha_j+\beta_j)h\left(2-\frac{(\alpha_j+\beta_j)^2}{z}\right)+ s\\
	&=(-\alpha_j\beta_j)^{\frac{r-1}{2}}(\alpha_j+\beta_j)h\left(2-\frac{(\alpha_j+\beta_j)^2}{\alpha_j\beta_j}\right)+s,
    \end{align*}
where we used the first part of \eqref{eqn:rootproduct}. By \cite[Lemma 2.14]{CelineClusters}, this expression equals
\[F(\gamma_j) =\alpha_j^r+\beta_j^r + s = 0,
\]
where the last equality follows by the second part of \eqref{eqn:rootproduct}.
\end{proof}

\begin{notation}The quantity $s^2-4z^{r}$ will appear many times during this paper. Consequently, and in order to lighten notation, we will denote it by $\Delta$. 
\end{notation}

Similarly, for $1 \le k \le r-1$, we shall use the notation $\gamma_{-k}$ to mean $\gamma_{r-k}$. With this notation and from the expression of the roots in Proposition \ref{prop:roots}, we note that 
 \begin{align}\label{eq:alpha-beta}
      \alpha_0^r-\beta_0^r=-\sqrt{\Delta}, \quad \quad \alpha_0^r+\beta_0^r=-s.
 \end{align}

\begin{remark}
\label{rmk:rootbasefield}
In Proposition  \ref{prop:roots} we chose $-\alpha_0$ to be an $r$-th root of $\frac{s+\sqrt{\Delta}}{2}\in\Q_q(\sqrt{\Delta})$. Similarly, we chose $-\beta_0$ to be an $r$-th root of $\frac{s-\sqrt{\Delta}}{2}\in\Q_q(\sqrt{\Delta})$ so that $\alpha_0\beta_0=z$. 
By commutative algebra arguments (see for example \cite[Section VI, Theorem 9.1]{Lang}), we know that the polynomial $x^r-\frac{s+\sqrt{\Delta}}{2}\in\Q_q(\sqrt{\Delta})[x]$ is reducible if and only if it has a root in $\Q_q(\sqrt{\Delta})$. In the case where it is, we adopt the convention that $\alpha_0 \in \Q_q(\sqrt{\Delta})$.
 \end{remark}

In order to build the cluster picture, we need to study the valuation of differences between distinct roots of $F(x)$. This is the content of the following lemma.

\begin{lemma}\label{lemma:differenceroots}
	For any $0\le j,k\le r-1$, the difference between two roots of $F(x)$ is given by 
	\begin{align*}
		\gamma_k-\gamma_j=\zeta_r^k(1-\zeta_r^{j-k})(\alpha_0-\zeta_r^{-k-j}\beta_0).
	\end{align*}
	For any fixed $0\le k\le r-1$, we have
	\begin{align*}
		\prod_{\substack{j=0 \\ j\neq k}}^{r-1}(\gamma_k-\gamma_j)=\frac{-r\sqrt{\Delta}}{\zeta_r^k(\alpha_0-\zeta_r^{-2k}\beta_0)}.
	\end{align*}
\end{lemma}
\begin{proof}
The first claim follows directly from Proposition \ref{prop:roots}. The product expression is a consequence of \eqref{eq:alpha-beta} and the fact that $\prod_{i=1}^{r-1} (1-\zeta_r^i) = r$.
\end{proof}

\section{\texorpdfstring{Cluster picture of $C(z,s)$}{Cluster picture of C(z,s)}}
\label{sec:clusterpictureC}
In this section, we aim to build the cluster pictures of the curve $C = C(z,s)$ defined in \eqref{eqn:C}, which we will use to compute its conductor in Section \ref{Sec:conductor}. We begin by computing its discriminant to determine the potential primes of bad reduction, which is almost immediate from Lemma \ref{lemma:differenceroots}.

\begin{lemma}
    \label{lemma:discriminant}
    The discriminant of the curve $C$ is given by 
    \[\Delta_C = (-1)^{\frac{r-1}{2}} 2^{2(r-1)} r^r \Delta^{\frac{r-1}{2}}.\]
\end{lemma}

\begin{proof}
The discriminant of $F(x)$ is given by
    \[\Delta_{F(x)} = (-1)^{\frac{r(r-1)}{2}} \prod_{k=0}^{r-1} \left(\prod_{\substack{j=0 \\ j\neq k}}^{r-1} (\gamma_k-\gamma_j)\right).    
    \] 
      Using Lemma \ref{lemma:differenceroots} and the expression of $\alpha_0^r-\beta_0^r$ in \eqref{eq:alpha-beta}, 
    we get $\Delta_{F(x)} = (-1)^{\frac{r-1}{2}}r^r \Delta^{\frac{r-1}{2}}$. Since $C$ has genus $(r-1)/2$ by Remark \ref{rmk:genus}, Definition \ref{def:discriminantcurve} gives that the discriminant of $C$ is simply $2^{2(r-1)}\Delta_{F(x)}$.
\end{proof}

From this lemma, we see that the only potential odd primes of bad reduction for $C$ are $q=r$ and $q \mid \Delta$. For the cluster picture of $C$, primes which simultaneously divide $\Delta$ and $s$ (and therefore $z$) will require separate treatment. In order to deal with them, we will make use of the following technical lemma.

\begin{lemma} \label{lemma:divisibility}
Let $q \in \Z$ be an odd prime such that $q\mid \Delta$ and $q\mid s$.  For all $0\le j\le r-1$,
\begin{align*}
    \min\{v_q(\alpha_0),v_q(\beta_0)\} = v_q(\alpha_0-\zeta_r^j\beta_0)= \frac{v_q(s)}{r} \quad \text{and} \quad v_q(\sqrt{\Delta})=v_q(s).
\end{align*}
\end{lemma}

\begin{proof}
The last claim follows directly from \ref{item1} and the definition of $\Delta$. Adding and substracting the expressions for $\alpha_0^r - \beta_0^r$ and $\alpha_0^r + \beta_0^r$ in \eqref{eq:alpha-beta}, we find that $v_q(\alpha_0), v_q(\beta_0)\ge {v_q(s)}/{r}$. Hence, it follows that $v_q(\alpha_0-\zeta_r^j\beta_0)\ge \min\{v_q(\alpha_0), v_q(\beta_0)\}\ge{v_q(s)}/{r}$ for all $0\le j\le r-1$. Finally, we note that 
\begin{align*}
    v_q(s)=v_q(-\sqrt{\Delta}) = v_q(\alpha_0^r-\beta_0^r)=\sum_{j=0}^{r-1}v_q(\alpha_0-\zeta_r^j\beta_0)\ge r\cdot\frac{v_q(s)}{r}=v_q(s),
\end{align*}
where the first equality follows by \ref{item1} and the second one follows by the first expression of \eqref{eq:alpha-beta}. Hence all inequalities are in fact equalities, proving the first claim. 
\end{proof}
With the previous results, we are finally able to compute the cluster picture for $C$ at each odd prime of bad reduction. Before that, {set the following notation.}

\begin{notation} \label{rem:root0}
Let $q$ be an odd prime such that $q \mid \Delta$ and $q \nmid s$ (so that $q \nmid z$). These conditions are equivalent to $v_q(\alpha_0^r-\beta_0^r) > 0$ and $v_q(\alpha_0) =v_q(\beta_0) = 0$. If $q\neq r$, then Lemma \ref{lemma:lematecnico} implies there is a unique $0\le j_0 \le r-1$ such that $v_q(\alpha_0-\zeta_r^{j_0}\beta_0)>0$. We denote by $i_0$ the unique integer satisfying $0\le i_0\le r-1$ and $-2i_0\equiv j_0\pmod r$. Recall that if $\zeta_r\notin \Q_q$, we have $j_0=0$ and so $i_0=0$. 
\end{notation}

\begin{thm}
\label{thm:clustersQ}
Let $q$ be an odd prime of bad reduction for $C/\mathbb{Q}_q$, so in particular $q\mid r\Delta$.
	\begin{enumerate}
	\item If $q\nmid rs$ then the cluster picture of $C/\Q_q$ is

            \[ \clusterpicture            
	\Root[A] {2} {first} {r1};
	\Root[A] {3} {r1} {r2};
	\ClusterLDName c1[][n][\gamma_0,\gamma_{2i_0}] = (r1)(r2);
        \Root[A] {6} {r2} {r3};
        \Root[A] {3} {r3} {r4};
        \ClusterLDName c2[][n][\gamma_1,\gamma_{2i_0-1}] = (r3)(r4);
	\Root[Dot] {8} {r4} {r5};
	\Root[Dot] {} {r5} {r6};
	\Root[Dot] {} {r6} {r7};
	\Root[A] {5} {r7} {r8};
        \Root[A] {3} {r8} {r9};
        \ClusterLDName c3[][n][\gamma_{\frac{r-1}{2}},\gamma_{2i_0 -\frac{r-1}{2}}] = (r8)(r9);
        \Root[A] {6} {r9} {r10};
        \ClusterLDName cc[][][\gamma_{i_0}] = (r10);
	\ClusterLDName c[][0][] = (c1)(c2)(r4)(r5)(r6)(c3)(cc);
	\endclusterpicture, \quad \text{ where } n=\frac{v_{{q}}(\Delta)}{2}. \]

    \item If $q\neq r$ and $q \mid s$ then the cluster picture of $C/\Q_q$ is

        \[ \clusterpicture            
	\Root[A] {1} {first} {r1};
	\Root[A] {3} {r1} {r2};
        \Root[A] {3} {r2} {r3};
	\Root[Dot] {5} {r3} {r4};
	\Root[Dot] {} {r4} {r5};
	\Root[Dot] {} {r5} {r6};
	\Root[A] {-9} {r7} {r8};
        \Root[A] {3} {r8} {r9};
        \Root[A] {3} {r9} {r10};
	\ClusterLDName c[][\frac{v_q(s)}{r}][] = (r1)(r2)(r3)(r4)(r5)(r6)(r7)(r8)(r9)(r10);
	\endclusterpicture \]

		\item If $q=r$ and $r\nmid \Delta$, then the cluster picture of $C/\Q_r$ is

        \[ \clusterpicture            
	\Root[A] {1} {first} {r1};
	\Root[A] {3} {r1} {r2};
        \Root[A] {3} {r2} {r3};
	\Root[Dot] {5} {r3} {r4};
	\Root[Dot] {} {r4} {r5};
	\Root[Dot] {} {r5} {r6};
	\Root[A] {-9} {r7} {r8};
        \Root[A] {3} {r8} {r9};
        \Root[A] {3} {r9} {r10};
	\ClusterLDName c[][\frac{1}{r-1}][] = (r1)(r2)(r3)(r4)(r5)(r6)(r7)(r8)(r9)(r10);
	\endclusterpicture \]

		\item If $q=r$,  $r\mid \Delta$ and $r \mid s$ then the cluster picture of $C/\Q_r$ is

        \[ \quad \quad \ \clusterpicture            
	\Root[A] {1} {first} {r1};
	\Root[A] {3} {r1} {r2};
        \Root[A] {3} {r2} {r3};
	\Root[Dot] {5} {r3} {r4};
	\Root[Dot] {} {r4} {r5};
	\Root[Dot] {} {r5} {r6};
	\Root[A] {-9} {r7} {r8};
        \Root[A] {3} {r8} {r9};
        \Root[A] {3} {r9} {r10};
	\ClusterLDName c[][\frac{1}{r-1}+\frac{v_r(s)}{r}][] = (r1)(r2)(r3)(r4)(r5)(r6)(r7)(r8)(r9)(r10);
	\endclusterpicture \]

         \item If $q=r$, $1 \le v_r(\Delta) \le 2$ and $r \nmid s$ then the cluster picture of $C/\Q_r$ is  
            
        \[ \quad \quad \ \clusterpicture            
	\Root[A] {1} {first} {r1};
	\Root[A] {3} {r1} {r2};
        \Root[A] {3} {r2} {r3};
	\Root[Dot] {5} {r3} {r4};
	\Root[Dot] {} {r4} {r5};
	\Root[Dot] {} {r5} {r6};
	\Root[A] {-9} {r7} {r8};
        \Root[A] {3} {r8} {r9};
        \Root[A] {3} {r9} {r10};
	\ClusterLDName c[][\frac{1}{r-1}+\frac{v_r(\Delta)}{2r}][] = (r1)(r2)(r3)(r4)(r5)(r6)(r7)(r8)(r9)(r10);
	\endclusterpicture \]
 
            \item If $q=r$,  $v_r(\Delta) \ge 3$ and $r \nmid s$
              then the cluster picture of $C/\Q_r$ is

        \[ \clusterpicture            
	\Root[A] {2} {first} {r1};
	\Root[A] {3} {r1} {r2};
	\ClusterLDName c1[][n][\gamma_1,\gamma_{-1}] = (r1)(r2);
        \Root[A] {8} {r2} {r3};
        \Root[A] {3} {r3} {r4};
        \ClusterLDName c2[][n][\gamma_2,\gamma_{-2}] = (r3)(r4);
	\Root[Dot] {8} {r4} {r5};
	\Root[Dot] {} {r5} {r6};
	\Root[Dot] {} {r6} {r7};
	\Root[A] {8} {r7} {r8};
        \Root[A] {3} {r8} {r9};
        \ClusterLDName c3[][n][\gamma_{\frac{r-1}{2}},\gamma_{\frac{r+1}{2}}] = (r8)(r9);
        \Root[A] {8} {r9} {r10};
        \ClusterLDName cc[][][\gamma_0] = (r10);
	\ClusterLDName c[][\frac{2}{r-1}][] = (c1)(c2)(r4)(r5)(r6)(c3)(cc);
	\endclusterpicture, \quad \text{ where } n=\frac{v_r(\Delta)}{2}-\frac{r}{r-1}. \]

	\end{enumerate}
\end{thm}
\begin{proof}

(1) Suppose $q \mid \Delta$ and $q \nmid s$. Let $i_0$ as in Notation \ref{rem:root0}. By Lemma \ref{rem:root0} we get that $v_q(\gamma_{i_0}-\gamma_j) = 0$ for all $0\le j \le r-1$ with $j\neq i_0$. Thus $\gamma_{i_0}$ is an isolated root and does not belong to any other cluster than $\mathcal{R}$. Take $k\neq i_0$. Since $q \nmid s$, we have that $q \nmid z = \alpha_0\beta_0$ and case (1) of 
Lemma \ref{lemma:lematecnico} yields that $v_q(\alpha_{0}-\zeta_r^{-k-j}\beta_0)>0$ if and only if $j\equiv 2i_0-k \pmod r$. Then by Lemma \ref{lemma:differenceroots} we have
\begin{equation*}
    v_q(\gamma_k-\gamma_{2i_0-k})=v_q\left(\prod_{\substack{j=0 \\ j\neq k}}^{r-1}(\gamma_k-\gamma_j)\right)=v_q\left(\frac{-r\sqrt{\Delta}}{\zeta_r^k(\alpha_0-\zeta_r^{-2k}\beta_0)}\right)=\frac{v_q(\Delta)}{2},
\end{equation*}
where the last steps used the fact that $q \neq r$, along with $k \neq i_0$.

\vspace{5pt}

(2) Assume that $q\mid s$ and $q\mid\Delta$. By Lemmas \ref{lemma:differenceroots} and \ref{lemma:divisibility}, we have that for all ${0\le k, j \le r-1}$,\[
    	v_q(\gamma_k-\gamma_j)=v_q(\zeta_r^k(1-\zeta_r^{j-k})(\alpha_0-\zeta_r^{-j-k}\beta_0))=\frac{v_q(s)}{r},
\]
giving the desired cluster.

\vspace{5pt}

(3) If $q=r$ and $r\nmid \Delta$ then \eqref{eq:alpha-beta} gives that $v_r(\alpha_0-\zeta_r^j\beta_0)=0$ for all $0\le j \le r-1$. Then Lemma \ref{lemma:differenceroots} yields that 
\[v_r(\gamma_k-\gamma_j)=v_r(1-\zeta_r^{j-k})=\frac{1}{r-1}.\]

\vspace{5pt}

(4) If $q = r$, $r \mid \Delta$ and $r \mid s$, the proof is analogous to point (2), with the only difference being that now we have that ${v_r(1-\zeta_r^{j-k})=1/(r-1)}$, so
\begin{align*}
    	v_r(\gamma_k-\gamma_j)=v_r((1-\zeta_r^{j-k})(\alpha_0-\zeta_r^{-j-k}\beta_0))=\frac{1}{r-1} + \frac{v_r(s)}{r},
\end{align*}
where we applied Lemma \ref{lemma:divisibility} once again.

\vspace{5pt}

(5) If $q = r$, $r \nmid s$ and $1 \le v_r(\Delta) \le 2$, let us see that we are in the second case of Lemma \ref{lemma:lematecnico2}. Suppose on the contrary that the first case holds. Then, $v_r(\alpha_0-\beta_0) > v_r(1-\zeta_r) > 0$ and 
\begin{equation}
\label{eqn:valuationaux}
v_r(\alpha_0-\zeta_r^{-j}\beta_0) = v_r(1-\zeta_r) = \frac{1}{r-1},
\end{equation}
for all $1 \le j \le r-1$. 
This implies that
\begin{align}\label{eq:valuation(alpha-beta)}
    v_r(\alpha_0-\beta_0)=v_r(\alpha_0^r-\beta_0^r)-\sum_{j=1}^{r-1}v_r(\alpha_0-\zeta_r^{-j}\beta_0)=v_r(\alpha_0^r-\beta_0^r)-1 = \frac{v_r(\Delta)}{2}-1.
\end{align}
Since $v_r(\Delta) \le 2$,
         \eqref{eq:valuation(alpha-beta)} means that $v_r(\alpha_0-\beta_0)={v_r(\Delta)}/{2}-1\le 0$, giving a contradiction. Thus, case (2) of Lemma \ref{lemma:lematecnico2} holds, i.e.
            \[v_r(\alpha_0 - \zeta_r^k\beta_0) = v_r(\alpha_0 - \zeta_r^j\beta_0)\]
        for all $0 \le k,j \le r-1$. Consequently,
        \[\frac{v_r(\Delta)}{2} = v_r(\alpha_0^r-\beta_0^r) = \sum_{i=0}^{r-1} v_r(\alpha_0 - \zeta_r^i\beta_0) = rv_r(\alpha_0-\beta_0),\]
        and so 
        \[v_r(\alpha_0-\zeta_r^i\beta_0)= \frac{v_r(\Delta)}{2r},\]
        for all $i = 0, 1, \dots, r-1$. Then Lemma \ref{lemma:differenceroots} yields that
        \[
			v_r(\gamma_k-\gamma_j)=v_r((1-\zeta_r^{j-k})(\alpha_0-\zeta_r^{-j-k}\beta_0))= \frac{1}{r-1} + \frac{v_r(\Delta)}{2r},
        \]
        for all $0 \le k,j \le r-1$, as we wanted to show.

\vspace{5pt}

        (6)  Since $r \ge 5$ and $v_r(\Delta) \ge 3$, it follows that $v_r(\alpha_0^r-\beta_0^r) = v_r(\Delta) > {r}/(r-1)=rv_r(1-\zeta_r)$, and so we are in case (1) of Lemma \ref{lemma:lematecnico2} by the pigeonhole principle. Then, \eqref{eqn:valuationaux} and \eqref{eq:valuation(alpha-beta)} are still valid and, consequently, Lemma \ref{lemma:differenceroots} yields that
\[v_r(\gamma_0-\gamma_j)=v_r((1-\zeta_r)(\alpha_0-\zeta_r^{-j}\beta_0))=2v_r(1-\zeta_r)=\frac{2}{r-1},\]
for any $1 \le j \le r-1$ and 
		\begin{align*}
			v_r(\gamma_k-\gamma_j)=v_r((1-\zeta_r^{j-k})(\alpha_0-\zeta_r^{-j-k}\beta_0))=\begin{cases}
				\frac{2}{r-1} & \text{if } j\not\equiv-k \pmod r,\\
				\frac{1}{r-1} + \frac{v_r(\Delta)}{2}-1 &  \text{if } j\equiv-k \pmod r
			\end{cases}
		\end{align*}
        for any $1 \le j < k \le r-1$.
		Therefore all the roots belong to a cluster having depth ${2}/({r-1})$, the root $\gamma_0$ is isolated, and there are ${(r-1)}/{2}$ twins each with relative depth $v_r(\Delta)/2-{r}/({r-1})$.
\end{proof}

For Diophantine applications, it will be useful to compute the conductor of the curve $C$ over the field $K = \Q(\omega_1)$, with ring of integers $\cO_K$. In order to do this, we shall first compute its cluster picture. If $\id{q}$ is a prime of $\cO_K$ above $q$, by standard local field arguments we have the isomorphism $K_\id{q}\simeq \Q_q(\omega_1)$. We will constantly use this identification when working with $K_\id{q}$.

\begin{coro}
\label{coro:clusteroverK}Let $\id{q}$ be an odd prime of $\cO_K$ of bad reduction for $C/K$ lying above the rational prime $q$, and let $\id{r}$ be the unique prime ideal of $\cO_K$ over $r$. Then
	\begin{enumerate}
	\item If $q\nmid rs$ then the cluster picture of $C/K_\id{q}$ is

        \[ \clusterpicture            
	\Root[A] {2} {first} {r1};
	\Root[A] {3} {r1} {r2};
	\ClusterLDName c1[][n][\gamma_0,\gamma_{2i_0}] = (r1)(r2);
        \Root[A] {6} {r2} {r3};
        \Root[A] {3} {r3} {r4};
        \ClusterLDName c2[][n][\gamma_1,\gamma_{2i_0-1}] = (r3)(r4);
	\Root[Dot] {8} {r4} {r5};
	\Root[Dot] {} {r5} {r6};
	\Root[Dot] {} {r6} {r7};
	\Root[A] {5} {r7} {r8};
        \Root[A] {3} {r8} {r9};
        \ClusterLDName c3[][n][\gamma_{\frac{r-1}{2}},\gamma_{2i_0 -\frac{r-1}{2}}] = (r8)(r9);
        \Root[A] {6} {r9} {r10};
        \ClusterLDName cc[][][\gamma_{i_0}] = (r10);
	\ClusterLDName c[][0][] = (c1)(c2)(r4)(r5)(r6)(c3)(cc);
	\endclusterpicture, \quad \text{ where } n=\frac{v_{{q}}(\Delta)}{2}. \]
    
.
    \item If $q\neq r$ and $q\mid s$ then the cluster picture of $C/K_{\id{q}}$ is

        \[ \clusterpicture            
	\Root[A] {1} {first} {r1};
	\Root[A] {3} {r1} {r2};
        \Root[A] {3} {r2} {r3};
	\Root[Dot] {5} {r3} {r4};
	\Root[Dot] {} {r4} {r5};
	\Root[Dot] {} {r5} {r6};
	\Root[A] {-9} {r7} {r8};
        \Root[A] {3} {r8} {r9};
        \Root[A] {3} {r9} {r10};
	\ClusterLDName c[][\frac{v_{{q}}(s)}{r}][] = (r1)(r2)(r3)(r4)(r5)(r6)(r7)(r8)(r9)(r10);
	\endclusterpicture \]
    
		\item If $q = r$ and $r\nmid \Delta$, then the cluster picture of $C/K_\id{r}$ is

        \[ \clusterpicture            
	\Root[A] {1} {first} {r1};
	\Root[A] {3} {r1} {r2};
        \Root[A] {3} {r2} {r3};
	\Root[Dot] {5} {r3} {r4};
	\Root[Dot] {} {r4} {r5};
	\Root[Dot] {} {r5} {r6};
	\Root[A] {-9} {r7} {r8};
        \Root[A] {3} {r8} {r9};
        \Root[A] {3} {r9} {r10};
	\ClusterLDName c[][\frac{1}{2}][] = (r1)(r2)(r3)(r4)(r5)(r6)(r7)(r8)(r9)(r10);
	\endclusterpicture \]

		\item If $q = r$, $r\mid \Delta$ and  $r\mid s$ then the cluster picture of $C/K_\id{r}$ is

        \[ \quad \quad \ \clusterpicture            
	\Root[A] {1} {first} {r1};
	\Root[A] {3} {r1} {r2};
        \Root[A] {3} {r2} {r3};
	\Root[Dot] {5} {r3} {r4};
	\Root[Dot] {} {r4} {r5};
	\Root[Dot] {} {r5} {r6};
	\Root[A] {-9} {r7} {r8};
        \Root[A] {3} {r8} {r9};
        \Root[A] {3} {r9} {r10};
	\ClusterLDName c[][\frac{1}{2}+\frac{(r-1)v_{r}(s)}{2r}][] = (r1)(r2)(r3)(r4)(r5)(r6)(r7)(r8)(r9)(r10);
	\endclusterpicture \]

            \item If $q=r$, $1 \le v_{r}(\Delta) \le 2$ and $r\nmid s$  then the cluster picture of $C/K_\id{r}$ is  
    
        \[ \quad \quad \ \clusterpicture            
	\Root[A] {1} {first} {r1};
	\Root[A] {3} {r1} {r2};
        \Root[A] {3} {r2} {r3};
	\Root[Dot] {5} {r3} {r4};
	\Root[Dot] {} {r4} {r5};
	\Root[Dot] {} {r5} {r6};
	\Root[A] {-9} {r7} {r8};
        \Root[A] {3} {r8} {r9};
        \Root[A] {3} {r9} {r10};
	\ClusterLDName c[][\frac{1}{2}+\frac{(r-1)v_{r}(\Delta)}{4r}][] = (r1)(r2)(r3)(r4)(r5)(r6)(r7)(r8)(r9)(r10);
	\endclusterpicture \]

            \item If $q=r$, $v_{r}(\Delta) \ge 3$ and $r\nmid s$
              then the cluster picture of $C/K_\id{r}$ is

        \[ \clusterpicture            
	\Root[A] {2} {first} {r1};
	\Root[A] {3} {r1} {r2};
	\ClusterLDName c1[][n][\gamma_1,\gamma_{-1}] = (r1)(r2);
        \Root[A] {8} {r2} {r3};
        \Root[A] {3} {r3} {r4};
        \ClusterLDName c2[][n][\gamma_2,\gamma_{-2}] = (r3)(r4);
	\Root[Dot] {8} {r4} {r5};
	\Root[Dot] {} {r5} {r6};
	\Root[Dot] {} {r6} {r7};
	\Root[A] {8} {r7} {r8};
        \Root[A] {3} {r8} {r9};
        \ClusterLDName c3[][n][\gamma_{\frac{r-1}{2}},\gamma_{\frac{r+1}{2}}] = (r8)(r9);
        \Root[A] {8} {r9} {r10};
        \ClusterLDName cc[][][\gamma_0] = (r10);
	\ClusterLDName c[][1][] = (c1)(c2)(r4)(r5)(r6)(c3)(cc);
	\endclusterpicture, \quad \text{ where } n=\frac{(r-1)v_{r}(\Delta)}{4}-\frac{r}{2}. \]

	\end{enumerate}
\end{coro}
 \begin{proof}
    We remark that, since $K$ is a subfield of $\Q(\zeta_r)$, it is unramified at any prime $q \neq r$ and totally ramified at $q = r$. Since the degree of the extension $K/\Q$ is ${(r-1)}/{2}$, it follows     that, for any $a \in K$, 
    \begin{align}\label{eq:idealvaluation}
    v_{\id{q}}(a)=\begin{cases}
    v_q(a) & \text{ if } \id{q} \nmid r, \\
    \frac{r-1}{2}v_r(a) & \text{ if } \id{q} = \id{r}.
    \end{cases} 
    \end{align}
   Then, the desired result is a direct consequence of Theorem \ref{thm:clustersQ}.
\end{proof}

\section{\texorpdfstring{Ramification index of the splitting field of $F(x)$}{Ramification index of the splitting field of F(x)}}
\label{Sec:ramification}

The behaviour of the conductor will be qualitatively different depending on whether $F(x)$ is reducible or not over $\Q_q$, so we first give conditions for the irreducibility of $F(x)$. This section follows very closely \cite[Section 4.3]{CelineClusters}. 

\subsection{\texorpdfstring{Criteria for the irreducibility of $F(x)$}{Criteria for the irreducibility of F(x)}}

We begin this section by stating the following result, which relates the reducibility of $F(x)$ with one of its roots and $\alpha_0$. 

\begin{prop} \label{prop:gamma0} The following are equivalent:
\begin{enumerate}
\item[(i)] There exists a root $\gamma_{k_0}\in \Q_q$.

 \item[(ii)]   $F(x)$ is reducible over $\Q_q$.

 \item[(iii)]  $x^r - \frac{s+\sqrt{\Delta}}{2}$ is reducible over $\Q_q(\sqrt{\Delta})$, and so $\alpha_0\in\Q_q(\sqrt{\Delta})$.
\end{enumerate}
 \end{prop}

 \begin{proof}  
 $(i) \Rightarrow (ii)$ is straightforward. Let us prove $(ii) \Rightarrow (iii)$ and $(iii) \Rightarrow (i)$.

 \vspace{3pt}

\noindent $(ii) \Rightarrow (iii)$: Consider the following towers of field extensions:
 \begin{equation}
    \label{eqn:tower1}
     \Q_q\subseteq^{d_0\mid 2}\Q_q(\sqrt{\Delta})\subseteq^{n_0 \le r}\Q_q(\sqrt{\Delta}, \alpha_0),
\end{equation}
and
\begin{equation}
\label{eqn:tower2}
\Q_q\subseteq^{n_0'\le r}\Q_q(\gamma_0)\subseteq^{d_0'\mid 2}\Q_q(\gamma_0, \alpha_0),
 \end{equation}
 where the superindices denote the degree of the extension. We claim that $\Q_q(\sqrt{\Delta}, \alpha_0) = \Q_q(\gamma_0, \alpha_0)$. Firstly, by \eqref{eqn:rootproduct}, we have that 
 \begin{equation}
 \label{eqn:beta0}
 \beta_0 = \frac{z}{\alpha_0} \in \Q_q(\sqrt{\Delta}, \alpha_0),
 \end{equation}
 and so $\gamma_0 = \alpha_0 + \beta_0 \in \Q_q(\sqrt{\Delta}, \alpha_0)$, and thus $\Q_q(\gamma_0, \alpha_0) \subseteq \Q_q(\sqrt{\Delta}, \alpha_0)$ . Conversely, by \eqref{eq:alpha-beta}, it follows that 
 \[\sqrt{\Delta} = \beta_0^r-\alpha_0^r = (\gamma_0-\alpha_0)^r - \alpha_0^r \in \Q_q(\gamma_0, \alpha_0),\]
 proving the equality of fields. Consequently, 
 \begin{equation}
 \label{eqn:productdegrees}
 d_0n_0 = n_0'd_0'.
 \end{equation}
 If $F(x)$ is reducible over $\Q_q$, it follows that $n'_0 < r$ and, from \eqref{eqn:productdegrees}, $n_0 < r$. From the definition of $\alpha_0$, this means that the polynomial
 \[x^r - \frac{s+\sqrt{\Delta}}{2} \in \Q_q(\sqrt{\Delta})[x]\]
 is reducible. 

 \vspace{3pt}
 
 \noindent $(iii)\Rightarrow (i)$: From Remark \ref{rmk:rootbasefield}, it follows that $\alpha_0 \in \Q_q(\sqrt{\Delta})$ and by \eqref{eqn:beta0}, $\beta_0 \in \Q_q(\sqrt{\Delta})$. Let $\sigma$ be the non-trivial element of $\Gal(\Q_q(\sqrt{\Delta})/\Q_q)$. Since $\alpha_0^r=(-s-\sqrt{\Delta})/2$, we have  $\sigma(\alpha_0)^r=(-s+\sqrt{\Delta})/2$ and so $\sigma(\alpha_0)$ is an $r$-th root of $(-s+\sqrt{\Delta})/2$. Therefore, $\sigma(\alpha_0)=\beta_0\zeta_r^{-j}$, for some $0\le j\le r-1$. Then, since
\[\alpha_0\beta_0\zeta_r^j=\alpha_0\sigma(\alpha_0)=\text{Norm}(\alpha_0)\in\Q_q,\] 
and $\alpha_0\beta_0=z\in\Q_q$, we have that $\zeta_r^j\in\Q_q$. Take $0\le k\le r-1$ such that $2k\equiv1\pmod r$. Abusing the notation, we have
\begin{align*}
\gamma_{kj} &= \zeta_r^{kj}\alpha_0+\zeta_r^{-kj}\beta_0\\
&=\zeta_r^{kj}\alpha_0+\zeta_r^{-j+kj}\beta_0\\
&=\zeta_r^{kj}(\alpha_0+\beta_0)\\
&=\zeta_r^{kj}(\alpha_0+\sigma(\alpha_0))=\zeta_r^{kj}\text{Trace}(\alpha_0)\in\Q_q,
\end{align*}
 where the second equality follows from the fact that $-j+kj\equiv-kj\pmod r$. Then, the result follows by taking $0\le k_0 \le r-1$ such that $k_0\equiv kj \pmod r$.
 \end{proof}
 \begin{remark}\label{rem:k0}
 Notice that, from the proof of Proposition \ref{prop:gamma0},  $k_0=0$ if $\zeta_r\notin\Q_q$.
 \end{remark}

In order to prove conditions for the irreducibility of $F(x)$ in certain cases, we will need the following elementary lemma.

\begin{lemma}
\label{lemma:srs}
    Suppose that $v_r(\Delta) \ge 2$ and that $r \nmid s$. Then $v_r(s^r-2^{r-1}s) \ge 2$.
\end{lemma}

\begin{proof}
    By Fermat's Little Theorem, it is clear that $s^r \equiv 2^{r-1}s \pmod{r}$, and so $s^r = 2^{r-1}s + ar$ for some $a \in \Z$. It then suffices to prove that $r \mid a$. We note that 
    \begin{equation}
    \label{eqn:congruence1}
    4^{r-1}\Delta = (2^{r-1}s)^2-4^rz^{r} = (s^r-ar)^2-4^rz^{r} \equiv s^{2r} - 2ars^r - 4^rz^{r} \pmod{r^2}.
    \end{equation}
    Now, since $r^2 \mid \Delta$, we have that $s^2 \equiv 4z^{r} \pmod{r^2}$ and therefore
    \begin{equation}
    \label{eqn:congruence2}
    s^{2r} \equiv (4z^r)^{r} \equiv 4^rz^{r} \pmod{r^2}, 
    \end{equation}
    where the last congruence follows by
one more application of Fermat's Little Theorem. Combining \eqref{eqn:congruence1}, \eqref{eqn:congruence2} and the fact that $r^2 \mid \Delta$, we get that $r^2 \mid 2ars^r$, which can only happen if $r \mid a$, thereby proving the lemma.
\end{proof}

Now, we are finally ready to prove the main result of this subsection, which, under the assumption that $q \mid \Delta$, completely characterizes when $F(x)$ is irreducible over $\Q_q$.

 \begin{lemma}\label{lemma:reducible}
 Suppose that $q \mid \Delta$. Then the following holds.
 \begin{enumerate}
 \item If $q\nmid s$ then $F(x)$ is irreducible over $\Q_q$ if and only if $q=r$ and $v_r(\Delta)\le 2$.
 \item If $q\mid s$ then $F(x)$ is irreducible over $\Q_q$ and either $r \nmid v_q(\alpha_0^r)$ or $r \nmid v_q(\beta_0^r)$. 
 \end{enumerate}

 \end{lemma}

 \begin{proof}

(1) Firstly, let us consider the case $q \neq r$. Let $\id{q}$ be the unique prime ideal of $\Q_q(\sqrt{\Delta})$ lying over $q$. Since $q \mid \Delta$, it follows that $\id{q}\mid \Delta = s^2-4z^{r}$ so that 
 \begin{equation}
 \label{eqn:congruencia}
 s\equiv\pm2 z^r = 2(\pm z)^r\pmod {\id{q}}.
 \end{equation}
 Now, let us consider the polynomial $f(x) \in \Q_q(\sqrt{\Delta})[x]$ given by
 \begin{equation}
 \label{eqn:auxiliarypoly}
 f(x) = x^r +\alpha_0^r = x^r - \frac{s+\sqrt{\Delta}}{2}.
 \end{equation}
From \eqref{eqn:congruencia}, it is clear that $f(\pm z) \equiv 0 \pmod{\id{q}}$. Since $q \nmid s$, it follows that $q \nmid z$ and we may apply Hensel's Lemma \cite[Chapter 4, Lemma 3.1]{Cassels} to conclude that $f(x)$ has a root in $\Q_q(\sqrt{\Delta})$. Then, $F(x)$ is reducible over $\Q_q$ by  Proposition \ref{prop:gamma0}.

\vspace{5pt}

If $q = r$, we need to study the polynomial $g(x) = f(x - \alpha_0^r)$. It is clear that $g(x)$ will be reducible if and only if $f(x)$ is reducible and, by Proposition \ref{prop:gamma0}, this will happen precisely when $F(x)$ is reducible. We expand $g(x)$ and get the expression 
\begin{equation}
\label{eqn:g}
g(x) = (x-\alpha_0^r)^r + \alpha_0^r = x^r + \sum_{i=1}^{r-1} \binom{r}{i}(-1)^ix^{r-i} \alpha_0^{ir} + (-\alpha_0^{r^2}+\alpha_0^r).
\end{equation}
Suppose that $1 \le v_r(\Delta) \le 2$. We claim that the polynomial $g(x)$ is Eisenstein and therefore irreducible. Clearly, it suffices to see that $-\alpha_0^{r^2}+\alpha_0^r$ is a uniformizer in $\Q_r(\sqrt{\Delta})$. Let $\id{r}$ be the unique prime ideal of $\Q_r(\sqrt{\Delta})$ lying over $r$. If $v_r(\Delta) = 2$, the extension $\Q_r(\sqrt{\Delta})/\Q_r$ is unramified and thus $v_\id{r}(\sqrt{\Delta}) = v_{r}(\sqrt{\Delta}) = v_r(\Delta)/2 = 1$. If $v_r(\Delta) = 1$, the extension $\Q_r(\sqrt{\Delta})/\Q_r$ is totally ramified and thus $v_\id{r}(\Delta) = 2v_r(\Delta) = 2$. In any of the two cases, it follows that $v_{\id{r}}(\sqrt{\Delta}) = 1$. From the expression of $\alpha_0$ in \eqref{eqn:alpha0beta0}, we find that 
\[ 2\alpha_0^r \equiv -s-\sqrt{\Delta} \pmod{\id{r}^2} \quad \text{and} \quad 2^r\alpha_0^{r^2} \equiv -s^r \pmod{\id{r}^2} ,
\]
and so 
\[2^r(-\alpha_0^{r^2} + \alpha_0^r) \equiv s^r-2^{r-1}s-2^{r-1}\sqrt{\Delta} \equiv -2^{r-1}\sqrt{\Delta} \pmod{\id{r}^2},\]
where the last congruence follows from Lemma \ref{lemma:srs}, if $v_r(\Delta) = 2$ and from the fact that $v_\id{r}(s^r-2^{r-1}s) = 2v_{r}(s^r-2^{r-1}s)\ge2$ if $v_r(\Delta) = 1$, by Fermat's Little Theorem. 

\vspace{5pt}

Finally, let us suppose that $v_r(\Delta) \ge 3$. If $v_r(\Delta)$ is odd, the extension $\Q_r(\sqrt{\Delta})/\Q_r$ is ramified and mimicking the previous proof yields that 
\begin{equation}
\label{eqn:congruencescubic}
2\alpha_0^r \equiv -s \pmod{\id{r}^3} \quad \text{and}  \quad 2^r\alpha_0^{r^2} \equiv -s^r \pmod{\id{r}^3} .
\end{equation}
By Lemma \ref{lemma:srs}, we have that $v_{\id{r}}(s^r-2^{r-1}s) = 2v_r(s^r-2^{r-1}s) \ge 4$ and therefore, \eqref{eqn:congruencescubic} directly implies that $v_{\id{r}}(\alpha_0^r-\alpha_0^{r^2}) \ge 3$. Since $r \nmid s$ and $r \mid \Delta$, we have that $\id{r} \nmid \alpha_0$ and so the valuation of all the coefficients $a_i$ of $g(x)$ with $1 \le i \le r-1$ is equal to 
\[v_{\id{r}}(a_i) = v_{\id{r}}\left(\binom{r}
{i}\right) = v_{\id{r}}(r) = 2.
\]
Consequently, the Newton polygon of $g(x)$ has two edges (the first one connecting the vertices $(0,v_{\id{r}}(2(\alpha_0^r-\alpha_0^{r^2})))$ and $(1, 2)$ and the second one connecting  $(1,2)$ and $(r,0)$). By \cite[Chapter 6, Theorem 3.1]{Cassels}, it then follows that $g(x)$ is reducible. 

If $v_r(\Delta) \ge 3$ is even, the extension $\Q_r(\sqrt{\Delta})/\Q_r$ is now unramified and so $v_r(\sqrt{\Delta}) = v_r(\Delta)/2 \ge 2$. This, together with the expression of $\alpha_0$ in \eqref{eqn:alpha0beta0} gives that
\[2\alpha_0^r \equiv -s \pmod{r^2} \quad \text{and} \quad 2^r\alpha_0^{r^2} \equiv -s^r \pmod{r^2},\]
and so
\[2^r(\alpha_0^{r^2}-\alpha_0^r) \equiv -s^r + 2^{r-1}s \pmod{r^2}.\]
Hence $v_r(-\alpha_0^{r^2}-\alpha_0^r) \ge 2$ by Lemma \ref{lemma:srs}
while $v_r(a_i) = v_r(r) = 1$ for $1 \le i \le r-1$. Then, the polynomial $g(x)$ is reducible by an identical argument to the preceding one. 

\vspace{5pt}

(2) Suppose for the sake of contradiction that $F(x)$ is reducible over $\Q_q$. Then, by Proposition \ref{prop:gamma0}, $\alpha_0 \in \Q_q(\sqrt{\Delta})$ and thus $\beta_0 = z/\alpha_0 \in \Q_q(\sqrt{\Delta})$. In addition, we note that, by Lemma \ref{lemma:divisibility}, $v_q(\Delta) = 2v_q(s)$ is even and, consequently, the extension $\Q_q(\sqrt{\Delta})/\Q_q$ is unramified. Thus, $v_q(\alpha_0)$ and $v_q(\beta_0)$ are positive integers and Lemma \ref{lemma:divisibility} yields that ${v_q(s) = r\min\{v_q(\alpha_0), v_q(\beta_0)\} \in r\mathbb{Z}}$. This is a contradiction with condition \ref{item2}, and, consequently $F(x)$ is irreducible over $\Q_q$. Note that condition \ref{item1} also implies that at least one of $v_q(\alpha_0^r)$ or $v_q(\beta_0^r)$ cannot be divisible by $r$.
 \end{proof}

    
 
\subsection{\texorpdfstring{Splitting field of $F(x)$}{Splitting field of F(x)}}
We now explicitly describe the splitting field of $F(x)$, as a polynomial defined over $\Q_q$, which will naturally be different depending on the reducibility of $F(x)$ over $\Q_q$. For this purpose, let $\phi(x,y)\in\Q_q[x]$ denote the polynomial
\[\phi(x,y) = \frac{x^r+y^r}{x+y} = x^{r-1} - x^{r-2}y + \dots - xy^{r-2} + y^{r-1}.
\]

\begin{lemma} 
    \label{lemma:rootsingamma0}
	Let $\alpha_0$ and $\beta_0$ be as in \eqref{eqn:alpha0beta0}. Then, $\phi(\alpha_0,-\beta_0)\in\Q_q(\gamma_0)$.
\end{lemma}
\begin{proof}
    Expanding the expression for $\phi(\alpha_0, -\beta_0)$, we get that 
    \begin{equation}
    \label{eqn:phialpha0beta0}
    \phi(\alpha_0, -\beta_0) = \sum_{i=1}^{r-1} \alpha_0^i\beta_0^{r-1-i}.
    \end{equation}
    Clearly, if $\alpha_0, \beta_0 \in \Q_q(\gamma_0)$, the statement of the lemma is true, so suppose that ${\alpha_0, \beta_0 \not \in \Q_q(\gamma_0)}$. 
    By \eqref{eqn:rootproduct} and the definition of $\gamma_0$, it follows that $\alpha_0$ and $\beta_0$ are roots of the polynomial
    \[(x-\alpha_0)(x-\beta_0) = x^2 - \gamma_0x + z \in \Q_q(\gamma_0)[x],\]
    so the extension $\Q_q(\alpha_0)/\Q_q(\gamma_0)$ is quadratic and the only non-trivial element $\sigma$ of its Galois group permutes $\alpha_0$ and $\beta_0$. It is then clear by the expression in \eqref{eqn:phialpha0beta0} that $\sigma$ leaves $\phi(\alpha_0, -\beta_0)$ invariant and, consequently, $\phi(\alpha_0, -\beta_0) \in \Q_q(\gamma_0)$.
\end{proof}

Recall from Section \ref{Sec:preliminaries} that we defined $\omega_j=\zeta_r^j+\zeta_r^{-j}$ for all $1\le j\le r-1$. To ease the notation, let $\omega:=\omega_1$. Similarly, we set  $\tau_j=\zeta_r^j-\zeta_r^{-j}$ and $\tau:=\tau_1$. Throughout the rest of the paper, we let $L$ be the splitting field of $F(x)$.

\begin{lemma}\label{lemma:inclusion}
Suppose $\zeta_r\in\Q_q$. Then $\Q_q(\gamma_{0})\subseteq\Q_q(\gamma_{j})$ for all $0\le j\le r-1$.
\end{lemma}
\begin{proof}
The statement follows from 
the equality $\gamma_0=(\gamma_{j}+\gamma_{-j})/{\omega_{j}}$.
\end{proof}
\begin{prop}\label{prop:splittingfield}
	The splitting field of $F(x)$ equals $L =\Q_q(\tau\sqrt{\Delta},\gamma_{0})$. Moreover, if $\zeta_r\in\Q_q$, $L=\Q_q(\tau\sqrt{\Delta},\gamma_{j})$ for all $0\le j \le r-1$.
\end{prop}
\begin{proof}
Set $L':=\Q_q(\tau\sqrt{\Delta},\gamma_{0})$ and let $L$ be the splitting field of the polynomial $F(x)$. To show that $L=L'$, we will prove the double inclusion.  Firstly, let us see that $L'\subseteq L$. By Proposition \ref{prop:roots}, $\gamma_0$ is a root of $F(x)$ and so $\gamma_0 \in L$. 
To see that $\tau\sqrt{\Delta}\in L$, we first note that 
\begin{equation}
\label{eqn:aux1}\tau_j(\alpha_0-\beta_0)=\gamma_j-\gamma_{-j}\in L.
\end{equation} On the other hand, by \eqref{eq:alpha-beta}, we know that $-\sqrt{\Delta}=\alpha_0^r-\beta_0^r$. This, combined with the fact that
	\begin{equation}
        \label{eqn:aux2}
        -\tau_j\sqrt{\Delta}=\tau_j(\alpha_0^r-\beta_0^r) = \tau_j(\alpha_0-\beta_0)\phi(\alpha_0,-\beta_0),
	\end{equation}
	implies that $\tau_j\sqrt{\Delta}\in L$ if and only if $\phi(\alpha_0,-\beta_0)\in L$. By Lemma \ref{lemma:rootsingamma0}, we have that $\phi(\alpha_0,-\beta_0)\in \Q_q(\gamma_0)\subseteq L$, thereby proving that $L' \subseteq L$. 

    \vspace{5pt}

    We now show that $L \subseteq L'$. It is not hard to see that  that $\gamma_j + \gamma_{-j} = \omega_j\gamma_0$.    This, combined with \eqref{eqn:aux1}, yields that 
    \begin{equation}
    \label{eqn:gammaj}
    \gamma_j = \frac{1}{2}\left(\omega_j\gamma_0 + \tau_j(\alpha_0-\beta_0)\right).
    \end{equation}
    We claim that both $\omega_j$ and $\tau_j(\alpha_0-\beta_0)$ belong to $L'$. Firstly, we see that 
    \begin{equation}
    \label{eqn:cuentatau}
    (\tau\sqrt{\Delta})^2 = \Delta(\omega_2-2) \quad \text{so that} \quad \omega_2 = \frac{(\tau\sqrt{\Delta})^2}{\Delta} + 2 \in L'.
    \end{equation}
    Since $\Q_q(\omega_j) = \Q_q(\omega_2)$ for any $1 \le j \le (r-1)/2$, it readily follows that $\omega_j \in L'$. Consequently, we have that $\omega_j, \gamma_0 \in L'$, and it is therefore sufficient to show that $\tau_j(\alpha_0-\beta_0) \in L'$. {Let us show that $\tau_j/\tau \in L'$. If $\zeta_r \in \Q_q$, this is clear since $\tau_j, \tau \in \Q_q \subseteq L'$. Otherwise, let $\sigma\in\Gal(\Q_q(\zeta_r)/\Q_q)$ with $\sigma(\zeta_r) = \zeta_r^{-1}$. Then,}
    $\tau_j/\tau$ is fixed by $\sigma$. By Galois theory, this shows that $\tau_j/\tau \in \Q_q(\omega) \subseteq L'$. Then, we have that 
    \[\tau_j\sqrt{\Delta} = \left(\frac{\tau_j}{\tau}\right)(\tau\sqrt{\Delta}) \in L',
    \]
    and by \eqref{eqn:aux2}, along with Lemma \ref{lemma:rootsingamma0}, we have that $\tau_j(\alpha_0-\beta_0) \in L'$, as we wanted to show. 

    \vspace{5pt}

    To finish the proof, it only remains to prove that $L=\Q_q(\tau\sqrt{\Delta},\gamma_{j})$ for each $j$, when $\zeta_r\in\Q_q$. Since $L=\Q_q(\tau\sqrt{\Delta},\gamma_{0})$, it is included in $\Q_q(\tau\sqrt{\Delta},\gamma_{j})$, by Lemma \ref{lemma:inclusion}. The other inclusion follows straightforward from the definition of $L$ and the fact that we already proved that $\tau\sqrt{\Delta}\in L$.
\end{proof}

\begin{remark}
    \label{rmk:splittingfield}
    We note that, during the proof of Proposition \ref{prop:splittingfield}, we actually showed that $\omega \in \Q_q(\tau\sqrt{\Delta})$. This is also true in \cite[Theorem 4.19]{CelineClusters}, but they do not need to use that fact at any point during their paper. 
    
    For us, this will be relevant (see for example, Proposition \ref{prop:gamma1}) and, while it is redundant, we will often write $L = \Q_q(\omega, \tau\sqrt{\Delta}, \gamma_0)$ to make explicit the fact that $\omega$ belongs to the splitting field of $F(x)$.
\end{remark}

For the remainder of the paper, let $E = \Q_q(\tau\sqrt{\Delta})$. From Propositions \ref{prop:gamma0} and \ref{prop:splittingfield}, it is clear that $E$ is the splitting field of $F(x)$ whenever $F(x)$ is reducible over $\Q_q$ and $\zeta_r\notin \Q_q$. In the following proposition, we show that if $q = r$ and $F(x)$ is reducible over $\Q_r$, $E$ is the field obtained by adjoining any root $\gamma_i\neq\gamma_0$ of $F(x)$.

\begin{prop} \label{prop:gamma1} Suppose that $q=r$ and $F(x)$ is reducible over $\Q_r$. Then $E=\Q_r(\gamma_1)$.
\end{prop}

\begin{proof}

Since $E$ is the splitting field of $F(x)$, $E/\Q_r$ is clearly a Galois extension, and we let $G$ denote its Galois group. From the definition of $\tau_j$, it is clear that $\tau_{r-j} = \tau_{-j} = -\tau_j$. In addition, an identical computation to \eqref{eqn:cuentatau} shows that 
\begin{equation}
\label{eqn:taujdelta}
\tau_j\sqrt{\Delta} = \pm \sqrt{\omega_{2j}-2}\sqrt{\Delta},
\end{equation}
where, by our prior discussion, $\tau_j$ and $\tau_{-j}$ have different signs. 
Now, let $S$ be the set of endomorphisms of $E$ defined by 
{\small
\[S =\left\{\sigma_{j, \pm}: E \to E \mid \sigma_{j, \pm}|_{\Q_r} = \text{id}_{\Q_r}, \sigma_{j,\pm}(\omega) = \omega_j \text{ and } \sigma_{j,\pm}(\tau\sqrt{\Delta}) = \tau_{\pm j}\sqrt{\Delta} \text{ with } 1\le j\le \frac{r-1}{2}\right\}.
\]
}
We claim that $G \subseteq S$. To see this, let $\sigma \in G$. Since $\Q_r(\omega)$ is a subfield of $E$, we have that $\sigma\mid_{ \Q_r(\omega)}$ needs to belong to $\Gal(\Q_r(\omega)/\Q_r)$ and, consequently, $\sigma(\omega) = \omega_j$ for some $j = 1, \dots, (r-1)/2$. From \eqref{eqn:cuentatau}, we see that
\[(\sigma(\tau\sqrt{\Delta}))^2 = \Delta(\sigma(\omega_2)-2) = (\omega_{2j}-2)\Delta,
\]
where the last equality follows since $\omega_{2j} = \omega_{j}^2-2$ and so $\sigma(\omega_2) = \sigma(\omega)^2-2 = \omega_j^2-2 = \omega_{2j}$. Thus, we have that 
\[\sigma(\tau\sqrt{\Delta}) = \pm \sqrt{\omega_{2j}-2}\sqrt{\Delta} = \tau_{\pm j}\sqrt{\Delta},
\]
where the last equality follows from \eqref{eqn:taujdelta} and the subsequent discussion. Therefore, ${\sigma = \sigma_{j, \pm} \in S}$, as desired. Next, we show that the set 
\[X = \{\sigma_{j, \pm}(\gamma_1) \mid \sigma_{j, \pm} \in S\}\]
has cardinality equal to $r-1$. By the expression of $\gamma_1$ in \eqref{eqn:gammaj} combined with \eqref{eqn:aux2}, we note that 
\[\sigma_{j, \pm}(\gamma_1) = \frac{1}{2}\left(\omega_j\gamma_0 - \frac{\tau_{\pm j}\sqrt{\Delta}}{\phi(\alpha_0, -\beta_0)}\right) = \frac{1}{2}\left(\omega_j\gamma_0 + \tau_{\pm j}(\alpha_0-\beta_0)\right),
\]
where we used the facts that, since $F(x)$ is reducible over $\Q_r$, both $\gamma_0$ and $\phi(\alpha_0, -\beta_0)$ lie in $\Q_r$ (see Proposition \ref{prop:gamma0}, Remark \ref{rem:k0} and Lemma \ref{lemma:rootsingamma0}). Now, since $\omega_j = \omega_{-j}$, it follows that $\sigma_{j, \pm}(\gamma_1) = \gamma_{\pm j}$, thereby proving that $\#X = r-1$. 

\vspace{5pt}

Finally, we note that, since $\Q_r(\omega) \subseteq E$, $[E:\Q_r]$ is divisible by $(r-1)/2$, and by \eqref{eqn:cuentatau}, it follows that either $[E:\Q_r] = r-1$ or $[E:\Q_r] = (r-1)/2$. In the first case, we have that $G = S$ and, since $\#X = r-1$, the set 
\[\mathcal{R}' = \{\gamma_1, \dots, \gamma_{r-1}\}\]
 has only one orbit under the action of $G$. Consequently, the polynomial $F(x)/(x-\gamma_0) \in \Q_r[x]$ is irreducible and so $[\Q_r(\gamma_1):\Q_r] = r-1$ and $E = \Q_r(\gamma_1)$. If $[E:\Q_r] = (r-1)/2$, $G$ is a subset of $S$ of size $(r-1)/2$ and an identical argument shows that $\mathcal{R}'$ has two orbits under the action of $G$. Therefore, there are two irreducible factors of $F(x)/(x-\gamma_0)$ of degree $(r-1)/2$ and it follows that $[\Q_r(\gamma_1):\Q_r] = (r-1)/2 = [E:\Q_r]$, whence $E = \Q_r(\gamma_1)$, as desired.
\end{proof}

\subsection{Ramification indices}

In this section, we shall compute the ramification index of $L$, and we will prove that the extension $L/\Q_q$ is wildly ramified precisely when $F(x)$ is irreducible over $\Q_q$. The following lemma will allow us to compute the wild conductor in those situations. 
For this, we recall that $\Q_q(\alpha_0^r)=\Q_q(\beta_0^r)=\Q_q(\sqrt{\Delta})$.

\begin{lemma}
    \label{lemma:valuationdiscriminants}
    Suppose that $F(x)$ is irreducible over $\Q_q$, and let $\delta = v_q(\Delta_{\Q_q(\gamma_0)/\Q_q})$. Then the following holds.
    \begin{enumerate}
        \item If $q \mid \Delta$ and $q \mid s$, then $\delta = rv_q(r) + r-1$.
        \item If $q = r$ and $r \nmid \Delta$, then $\delta = r$.
        \item If $q = r$, $r \nmid s$ and $v_r(\Delta) = 1$, then $\delta=\frac{3r-1}{2}$.
        \item If $q = r$, $r \nmid s$ and $v_r(\Delta) = 2$, then $\delta=r$.
\end{enumerate}
\end{lemma}

\begin{proof}

    Throughout all cases of this proof and in order to lighten notation, we will let ${H = \Q_q(\sqrt{\Delta}) = \Q_q(\alpha_0^r)}$ and $M = \Q_q(\alpha_0)$. In addition, we will make extensive use of the following identity for the growth of the discriminant in field extensions $K_1 \subseteq K_2 \subseteq K_3$, which is \cite[Corollary 2.10]{Neukirch}:
    \begin{equation}
        \label{eqn:discriminantgrowth}
        \Delta_{K_3/K_1} = \Delta_{K_2/K_1}^{[K_3:K_2]} \norm_{K_2/K_1}\left(\Delta_{K_3/K_2}\right).
    \end{equation}
    Finally, we shall also exploit the fact that if a quadratic extension of local fields $K_2/K_1$ is ramified and $v_\id{q}(\cdot)$ is the normalized valuation on $K_1$, then $v_{\id{q}}(\Delta_{K_2/K_1}) = 1$. 

    \vspace{5pt}

    (1) Since $q \mid \Delta$ and $q \mid s$, it follows that $q \mid z = \alpha_0\beta_0$ and Lemma \ref{lemma:divisibility} yields that $v_q(\Delta)$ is even, so the extension $H/\Q_q$ is unramified.  In addition, by Lemma \ref{lemma:reducible}, either $r \nmid v_q(\alpha_0^r)$ or $r \nmid v_q(\beta_0^r)$. Without loss of generality, we shall suppose that $r \nmid v_q(\alpha_0^r)$. 

    Then, there exists $k \in \Z$ such that $kv_q(\alpha_0^r)\equiv 1\pmod r$. Equivalently, there exists $a \in \Z$ with $kv_q(\alpha_0^r)=ar+1$. Let $\pi={\alpha_0^k}/{q^a}\in M$. Then, we check that $v_q(\pi)={1}/{r}$ and, consequently, the polynomial $P(x)=x^r-\pi^r \in H[x]$ is Eisenstein. By \cite[Chapter 1, Theorem 1]{Frohlich}, $\cO_M=\cO_H[\pi]$ and  $\Delta_{M/H}$ is equal to the discriminant of $P(x)$, up to units. As a result, we find that 
    \[v_q(\Delta_{M/H}) = rv_q(r) + r-1.
    \]
    If $\sqrt{\Delta} \in \Q_q$, $H = \Q_q$ {and $M = \Q_q(\gamma_0)$}, so the previous expression directly gives the value of $\delta$ in the statement of the lemma. If $\sqrt{\Delta} \not \in \Q_q$, we see that the extension $H/\Q_q$ is unramified and quadratic, and then \eqref{eqn:discriminantgrowth} gives that $v_q(\Delta_{M/\Q_q}) = 2(rv_q(r) + r-1)$. By studying the tower of field extensions 
    \begin{equation}
        \label{eqn:towerfields}
    \Q_q \subseteq^{r} \Q_q(\gamma_0) \subseteq^{2} M,
    \end{equation}
    we note that the extension $M/\Q_q(\gamma_0)$ is unramified (since the ramification degree of $M/\Q_q$ is $r$ by our previous argument) and, consequently, $v_q(\norm_{\Q_q(\gamma_0)/\Q_q}(\Delta_{M/\Q_q(\gamma_0)})) = 0$. Therefore, \eqref{eqn:discriminantgrowth} yields that 
    \[v_q(\Delta_{\Q_q(\gamma_0)/\Q_q}) = \frac{v_q(\Delta_{M/\Q_q})}{2} = rv_q(r) + r-1.\]

    (2) If $r \nmid \Delta$, we see that the extension $H/\Q_r$ is unramified and from \eqref{eq:alpha-beta} we get that either $r \nmid \alpha_0$ or $r \nmid \beta_0$, so either $\alpha_0$ or $\beta_0$ is a unit. From \cite[Theorem 4.21]{CelineClusters}, we see that $v_r(\Delta_{M/H}) = r$ and we check that $v_r(\Delta_{\Q_r(\gamma_0)/\Q_r}) = r$ by using \eqref{eqn:discriminantgrowth} as in case (1). 

    \vspace{5pt}

    (3) Now, if $r \nmid s$ and $v_r(\Delta) = 1$, we have that $\sqrt{\Delta} \not \in \Q_r$ and the extension $H/\Q_r$ is ramified of degree $2$. Also, note that we found an Eisenstein polynomial for the extension $M/H$ in the proof of case (1) of Lemma \ref{lemma:reducible}, which was a translation of the polynomial $f(x) = x^r + \alpha_0^r$. In particular, if we set $\id{r}$ to be the unique prime ideal of $H$ above $r$, this means that 
    \[v_\id{r}(\Delta_{M/H}) = v_\id{r}(\Delta_{f(x)}) = 2r.
    \]
    Applying \eqref{eqn:discriminantgrowth}, we find that 
        \[v_r(\Delta_{M/\Q_r}) = rv_r\left(\Delta_{H/\Q_r}\right) + v_{\id{r}}\left(\Delta_{M/H}\right) = rv_r(\Delta_{H/\Q_r})+2r = 3r,        
        \]
        since the extension $H/\Q_r$ is totally ramified and, consequently, 
        \[v_r\left(\norm_{H/\Q_r}\left(\Delta_{M/H}\right)\right) = v_\id{r}\left(\Delta_{M/H}\right).
        \]
        By using the same formula 
for the tower extension in \eqref{eqn:towerfields}, along with the fact that the extension $M/\Q_q(\gamma_0)$ is totally ramified, we find that 
    \[v_r\left(\Delta_{M/\Q_r}\right) = 2v_r(\Delta_{\Q_r(\gamma_0)/\Q_r})+v_\id{r}(\Delta_{M/\Q_r(\gamma_0)}) = 2v_r(\Delta_{\Q_r(\gamma_0)/\Q_r}) + 1,\]
    and so $v_r(\Delta_{\Q_r(\gamma_0)/\Q_r}) = (3r-1)/2$, as desired. \\

    (4) The proof is identical to part (2), since the extension $H/\Q_r$ is unramified and either $\alpha_0$ or $\beta_0$ is a unit.
\end{proof}

\begin{coro}
\label{coro:totallyramified}
{Suppose that $F(x)$ is irreducible over $\Q_q$. Then, both $\Q_q(\alpha_0)/\Q_q(\alpha_0^r)$ and $\Q_q(\gamma_0)/\Q_q$ are totally ramified extensions of degree $r$.}
\end{coro}

\begin{proof}
    Since the polynomial $F(x)$ is irreducible, the extension $\Q_q(\gamma_0)/\Q_q$ has degree $r$. In addition, it is ramified since, by Lemma \ref{lemma:valuationdiscriminants}, the discriminant of the extension is divisible by $q$. Since $r$ is prime, the extension is necessarily totally ramified. Finally, the same fact is true about the extension $\Q_q(\alpha_0)/\Q_q(\alpha_0^r)$, by studying the tower of field extensions $\Q_q \subseteq^2 \Q_q(\alpha_0^r) \subseteq^r \Q_q(\alpha_0)$ and comparing it to \eqref{eqn:towerfields}.
\end{proof}

Now, we may finally proceed to compute the ramification index of $L/\Q_q$. We begin by first studying ramification of the field $E$, as follows. 

\begin{lemma} \label{lemma:indexram}
Let $E=\Q_q(\omega,\tau\sqrt{\Delta})$. The ramification index $e_{E/\Q_q(\omega)}$ of $E/\Q_q(\omega)$ is des\-cribed as follows.
\[
\begin{array}{|c||l||l|}
\hline
 q=r &\text{Condition} & e_{E/\Q_q}  \\ \hline
\times & 2\mid v_q(\Delta) & 1 \\ \hline
\times & 2\nmid v_q(\Delta) & 2\\ \hline
\checkmark & 2\mid\frac{r-1}{2}v_r(\Delta) & 2\\ \hline
\checkmark & 2\nmid\frac{r-1}{2}v_r(\Delta) & 1\\ \hline
\end{array}
\]

\end{lemma}

\begin{proof}
The minimal polynomial of $E$ over $\Q_q(\omega)$ divides
\begin{equation}
\label{eqn:definingpoly}
x^2 - \Delta\tau^2 = x^2 + \Delta(4 - \omega^2) \in \Q_q(\omega)[x].
\end{equation}
In particular, this implies that $e_{E/\Q_q(\omega)}\le 2$. Moreover, by \cite[Theorem 6.1]{Cassels},  $E/\Q_q(\omega)$ will be unramified if and only if $v_\mathfrak{q}(-4\Delta(\omega^2-4))$ is even, where $\id{q}$ is the unique prime in $\Q_q(\omega)$ above $q$, and $v_\id{q}$ its respective valuation. Then, the result follows from equation (\ref{eq:idealvaluation}) and the fact that $v_\id{q}(\omega^2-4)=\delta_{q,r}$, where $\delta$ is the Kronecker delta.
\end{proof}

\begin{thm}
\label{thm:ramification}
The ramification index $e_{L/\Q_q}$ of $L/\Q_q$ is equal to 
\begin{equation}\label{eq:ramindex}
e_{L/\Q_q}=\varepsilon_1\varepsilon_2e_{E/\Q_q(\omega)},
\end{equation}
where 
\[\varepsilon_1=\begin{cases}
\frac{r-1}{2} & \text{if } q=r,\\
1  & \text{if } q\neq r
\end{cases}
\quad \text{ and } \quad 
\varepsilon_2=\begin{cases}
r & \text{if } F(x) \text{ is irreducible over } \Q_q,\\
1  & \text{if } F(x) \text{ is reducible over } \Q_q.
\end{cases}\]
\end{thm}
\begin{proof}
Let $\varepsilon_1:=e_{\Q_q(\omega)/\Q_q}$ and $\varepsilon_2:=e_{L/E}$. Then, it is clear that equation (\ref{eq:ramindex}) holds. To prove the formula for the values of $\varepsilon_1$ note that the extension $\Q_q(\omega)/\Q_q$ is unramified (resp.\ totally ramified of degree $(r-1)/2$) if $q\neq r$ (resp.\ $q=r$).


Suppose $F(x)$ is reducible over $\Q_q$.  By Proposition \ref{prop:gamma0}, there exists $0\le k_0 \le r-1$ such that $\gamma_{k_0}\in\Q_q$. If $\zeta_r\notin\Q_q$ then $k_0=0$, by Remark \ref{rem:k0}. Then, Proposition \ref{prop:splittingfield} implies $L=E$. On the other hand, if $\zeta_r\in\Q_q$  we also have $L=E$ by Proposition  \ref{prop:splittingfield}, taking $j=k_0$ in the statement. 

Therefore,  it only remains to consider the case where $F(x)$ is irreducible, to compute $e_{L/E}$. In such a case, by Corollary~\ref{coro:totallyramified} we have that $e_{\Q_q(\alpha_0)/\Q_q(\alpha_0^r)}=r$.
Using the chain of field extensions
\[\Q_q\subseteq\Q_q(\alpha_0^r)\subseteq\Q_q(\alpha_0)\subseteq\Q_q(\zeta_r,\alpha_0)\]
we get that $r$ divides the ramification index $\Q_q(\zeta_r,\alpha_0)/\Q_q$. On the other hand, using the chain of field extensions 
\[\Q_q\subseteq E\subseteq L\subseteq \Q_q(\zeta_r,\alpha_0)\]
we get that $r\mid e_{L/E}$, since $r$ is prime and does not divide the ramification index of $E/\Q_q$ (less than or equal to $r-1$) neither that of $\Q_q(\zeta_r,\alpha_0)/L$ (less than or equal to $2$). Therefore, since $e_{L/E}\le r$, we get $e_{L/E}=r$.
\end{proof}

\section{\texorpdfstring{Conductor of $C(z,s)$}{Conductor of C(z,s)}}
\label{Sec:conductor}

In this section, we finish the computation of the conductor. We also performed comprehensive numerical verification, utilizing the computer algebra package \verb|Magma| \cite{Magma}, for  conductor computations, along with Tim Dokchitser’s cluster picture implementation. All codes and computations are available in \cite{github}.

\vspace{5pt}

Let $C=C(z,s)$ be the curve defined in \eqref{eqn:C}, where $z$ and $s$ satisfy the conditions \ref{item1} and \ref{item2}.  As in Definition \ref{def:conductor}, we denote the conductor exponent of $C/\Q_q$ as $\mathfrak{n}(C/\Q_q)$. Firstly, we will compute the tame part of the conductor.

\begin{prop} \label{prop:conductorFred} Let $q$ be an odd prime such that $q\mid rs$. The tame component is equal to  $\mathfrak{n}_\text{tame}(C/\Q_q)=r-1$. In particular,  if $F(x)$ is reducible over $\Q_q$, then $\mathfrak{n}(C/\Q_q)=r-1$.
\end{prop}

\begin{proof}

First of all, note that the last part of the statement follows from the fact that if $F(x)$ is reducible over $\Q_r$, Theorem \ref{thm:ramification} gives that $r$ does not divide $e_{L/\Q_r}$ and therefore the extension $L/\Q_r$ is tame.

 To compute the tame conductor, we will apply Theorem \ref{thm:formulascluster}, where $g=(r-1)/2$, by Remark \ref{rmk:genus}. The main idea is that to prove the proposition, it is sufficient to prove that $\#U/I_{\Q_q}=\#V/I_{\Q_q}$, and we shall do so in all cases where $q \mid rs$.

\vspace{3pt}

\textbf{(1) Case} $\boldsymbol{q \neq r}$: By the assumptions in the proposition, we have that  $q \mid s$. From case (2) of Theorem \ref{thm:clustersQ}, we see that the only odd clusters, besides $\mathcal{R}$, are the singletons $\{\gamma_i\}$ for $0\le i \le r-1$ and it is clear that $P(\{\gamma_i\})$ is the whole cluster $\mathcal{R}$, with depth $d_\mathcal{R} = {v_q(s)}/{r}$ and $\tilde{\lambda}_\mathcal{R} = {v_q(s)}/{2}$. Since the whole cluster $\mathcal{R}$ is stable under the action of the inertia group $I_{\Q_q}$, we compute
\begin{equation}
\label{eqn:xilambda}\xi_\mathcal{R}(\tilde{\lambda}_\mathcal{R}) = \max\left\{-v_2\left(\frac{v_q(s)}{2}\right), 0\right\}
= \begin{cases} 1 & \text{ if } v_q(s) \text{ is odd,} \\
0 & \text{ if } v_q(s) \text{ is even.} 
\end{cases}
\end{equation}
Since $d_\mathcal{R} = v_q(s)/r$, and $r$ is odd, we always find that $\xi_\mathcal{R}(d_\mathcal{R}) = 0$.
Consequently, we conclude that 
\[U = \begin{cases}
\emptyset & \text{ if } v_q(s) \text{ is odd}, \\
\{\{\gamma_i\} \colon 0\le i \le r-1\} & \text{ if } v_q(s) \text{ is even}. \\
\end{cases}
\]
In order to compute the set $V$, we note that the only proper non-übereven cluster is $\mathcal{R}$ and we have already computed $\xi_\mathcal{R}(\tilde{\lambda}_\mathcal{R})$ in \eqref{eqn:xilambda}. We therefore conclude that 
\[V = \begin{cases}
\emptyset & \text{ if } v_q(s) \text{ is odd}, \\
\{\mathcal{R}\} & \text{ if } v_q(s) \text{ is even}. \\
\end{cases}
\]
If $v_q(s)$ is odd, then $U=V=\emptyset$. Let us consider now the case where $v_q(s)$ is even. In this situation, it is clear that $\#(V/I_{\Q_q}) = 1$. 

By case (2) of Lemma \ref{lemma:reducible}, $F(x)$ is irreducible and so, by Corollary \ref{coro:totallyramified}, it follows that 
$e_{\Q_q(\gamma_0)/\Q_q} = r$. We note that this ramification index equals the size of orbits under the action of the inertia group $I_{\Q_q}$. Consequently,  the $r$ roots of the polynomial $F(x)$ lie under a single orbit under the action of the inertia group $I_{\Q_q}$, so that $\#(U/I_{\Q_q}) = 1 = \#(V/I_{\Q_q})$. 

\vspace{5pt}

From now on we will consider the case $q=r$.

\vspace{5pt}

      \textbf{(2) Case} $\boldsymbol{ r \nmid \Delta}$:  Note that the only odd clusters different from $\mathcal{R}$ are the singletons $\{\gamma_i\}$, which have $P(\gamma_i) = \mathcal{R}$. We compute $\tilde{\lambda}_R = {r}/{(2(r-1))}.$
This, together with the fact that $I_{\Q_r}$ stabilizes the cluster $\mathcal{R}$, yields that
\begin{equation}
\label{eqn:cuenta}
\xi_\mathcal{R}(\tilde{\lambda}_\mathcal{R}) = v_2(r-1) + 1 \neq 0,
\end{equation}
while $\xi_\mathcal{R}(d_\mathcal{R}) = \xi_\mathcal{R}\left({1}/({r-1})\right) = v_2(r-1)$, 
so that $\xi_\mathcal{R}(\tilde{\lambda}_\mathcal{\mathcal{R}}) > \xi_\mathcal{R}(d_R)$ and, consequently, the set $U$ is empty. On the other hand, the only proper non-übereven cluster is $\mathcal{R}$ and \eqref{eqn:cuenta} gives that $V = \emptyset$. 

\vspace{3pt}

 \textbf{(3) Case} $\boldsymbol{r \mid \Delta, \ r \mid s}$: Once again,  the only odd clusters different to $\mathcal{R}$ are the singletons, with $P(\{\gamma_i\}) = \mathcal{R}$. Since $d_\mathcal{R}={1}/{(r-1)} + {v_r(s)}/{r}$ and $\tilde{\lambda}_\mathcal{R} = {r}/{(2(r-1))} + {v_r(s)}/{2}$, we have
\[\xi_\mathcal{R}(\tilde{\lambda}_\mathcal{R}) = v_2(r-1) + 1> v_2(r-1)=\xi_\mathcal{R}(d_\mathcal{R})>0,\]
so $U = V = \emptyset$.

\vspace{3pt}

\textbf{(4) Case} $\boldsymbol{r \mid \Delta, \ r \nmid s}$: In this case, we need to distinguish two subcases, depending on whether $v_r(\Delta) \le 2$ or $v_r(\Delta) \ge 3$.
\vspace{5pt}

\textbf{(4.1) Case} $\boldsymbol{1 \le v_r(\Delta) \le 2}$: In this case, the cluster picture is given by case (5) of Theorem \ref{thm:clustersQ}. We find that the only odd clusters other than $\mathcal{R}$ are the singletons $\{\gamma_i\}$, with parent cluster $\mathcal{R}$. Once again, we compute $\tilde{\lambda}_\mathcal{R} = {r}/{(2(r-1))} + {v_r(\Delta)}/{4}.$
Similarly, we may see that 
\begin{equation}
\label{eqn:cuentainter}
\xi_\mathcal{R}(d_\mathcal{R}) = \xi_\mathcal{R}\left(\frac{1}{r-1} + \frac{v_r(\Delta)}{2r}\right). 
\end{equation}
If either $r \equiv 1 \pmod{4}$ or $v_r(\Delta) = 2$, it readily follows that 
\[\xi_\mathcal{R}(\tilde{\lambda}_\mathcal{R}) = 1+v_2(r-1) > v_2(r-1) = \xi_\mathcal{R}(d_\mathcal{R}),
\] 
so $U = V = \emptyset$.
\vspace{5pt}

Now, suppose that $r \equiv 3 \pmod{4}$ and that $v_r(\Delta) = 1$. Then, $\tilde{\lambda}_\mathcal{R} = {(3r-1)}/{(4(r-1))}$.
If $r \equiv 7 \pmod {8}$, we may see that $\xi_\mathcal{R}(\tilde{\lambda}_\mathcal{R}) = 1$
while from \eqref{eqn:cuentainter} we see that $\xi_\mathcal{R}(d_\mathcal{R}) = 0$,
so that, once again, $U = V = \emptyset$.

Finally, the case where $r \equiv 3 \pmod{8}$ and $v_r(\Delta) = 1$ remains. In this situation, ${\xi_\mathcal{R}(\tilde{\lambda}_\mathcal{R}) = \xi_\mathcal{R}(d_\mathcal{R}) = 0}$,
so that 
\[U = \{\{\gamma_i\} \mid 0 \le i \le r-1\} \quad \text{ and } \quad V = \{\mathcal{R}\}.\]
Since $F(x)$ is irreducible (by Lemma \ref{lemma:reducible}), Corollary \ref{coro:totallyramified} yields that the ramification index of $\Q_r(\gamma_0)/\Q_r$ is $r$. In particular, all orbits of $\mathcal{R}$ under the action of inertia have size $r$. Consequently, we have that $\#U/I_{\Q_r} = 1 = \#V/I_{\Q_r}$.

\vspace{5pt}
\textbf{(4.2) Case} $\boldsymbol{v_r(\Delta) \ge 3}$: In this case the cluster picture is given by case (6) of Theorem \ref{thm:clustersQ}. Firstly, let $\mathfrak{t}_i$ denote the twin given by $\mathfrak{t}_i = \{\gamma_i, \gamma_{r-i} \}$ for $1\le i \le (r-1)/{2}$. By Lemma \ref{lemma:reducible}, $F(x)$ is reducible, and so $L = E$. Then, by Theorem \ref{thm:ramification}, the ramification degree of $E/\Q_r$ is $r-1$ or $(r-1)/2$. In addition, from Proposition \ref{prop:gamma1}, we have that $E = \Q_r(\gamma_1)$ and so the size of the orbits under the action of $I_{\Q_r}$ is either $r-1$ or ${(r-1)}/{2}$. In either case, we have that the inertia group is tame and that no $\mathfrak{t}_i$ is an orphan, and consequently, \cite[Theorem 1.3(iv)]{Bisatt} yields that 
\begin{equation}
\label{eqn:inertiadegree}[I_{\Q_r}:I_{\mathfrak{t}_i}] = \text{denom } d_\mathcal{R} = \frac{r-1}{2}, 
\end{equation}
where we used the fact that $P(\mathfrak{t}_i) = \mathcal{R}$, along with $d_\mathcal{R} = 2/(r-1)$. 
By case (6) of Theorem \ref{thm:clustersQ}, we have that 
\begin{equation*}
\label{eqn:twins}
d_{\mathfrak{t}_i} = \frac{v_r(\Delta)}{2} - \frac{r-2}{r-1} \quad \text{ and } 
\quad \tilde{\lambda}_{\mathfrak{t}_i} = \frac{v_r(\Delta)}{2}.
\end{equation*}
This, together with \eqref{eqn:inertiadegree}, gives that 
\[\xi_{\mathfrak{t}_i}(\tilde{\lambda}_{\mathfrak{t}_i}) = \begin{cases}
0 & \text{ if } 2 \mid \frac{r-1}{2} v_r(\Delta), \\
1 & \text{otherwise},
\end{cases}
\quad \text{ and } \quad
\xi_{\mathfrak{t}_i}(\tilde{d}_{\mathfrak{t}_i}) = \begin{cases}
1 & \text{ if } 2 \mid \frac{r-1}{2} v_r(\Delta), \\
0 & \text{{otherwise}}.
\end{cases}
\]
Finally, we compute the same quantities for the whole cluster $\mathcal{R}$. We get $\tilde{\lambda}_\mathcal{R} = {r}/{(r-1)}$, $\quad\xi_{\mathcal{R}}(\tilde{\lambda}_{\mathcal{R}}) = v_2(r-1)$ and $\xi_{\mathcal{R}}(d_{\mathcal{R}}) = v_2(r-1)-1$.

We note that the only odd clusters different to $\mathcal{R}$ are the singletons $\{\gamma_i\}$, while the only proper non-übereven clusters are the twins $\mathfrak{t}_i$ and the whole cluster $\mathcal{R}$. Since $P(\{\gamma_0\}) = \mathcal{R}$, $P(\{\gamma_i\}) = \mathfrak{t}_i$ for $1 \le i \le {(r-1)}/{2}$ and $P(\{\gamma_i\}) = \mathfrak{t}_{r-i}$ for $ {(r+1)}/{2} \le i \le r-1$, our computations above show that  $U = V = \emptyset$ if $2 \nmid {(r-1)}v_r(\Delta)/2$,
and that 
\[U = \{\{\gamma_i\} \mid 1 \le i \le r-1 \}, \quad V = \left\{\mathfrak{t}_i \text{ }\Big|\text{ } 1\le i \le \frac{r-1}{2}\right\}, \quad \text{ if } 2\text{ }\Big|\text{ }\frac{r-1}{2}v_r(\Delta).
\]
The result now follows by noting that when $2\mid ({r-1})v_r(\Delta)/2$, Theorem \ref{thm:ramification} gives that the ramification index of $E/\Q_r$ is $r-1$. Since, by Proposition \ref{prop:gamma1}, $E = \Q_r(\gamma_1)$, this ramification index equals the size of the orbits under the action of inertia and so all the roots $\gamma_i\neq\gamma_0$ lie in a single orbit under the action of $I_{\Q_r}$. Consequently, we have that $\#U/I_{\Q_r} = \#V/I_{\Q_r} = 1$.
\end{proof}

\begin{thm}
\label{thm:conductor} Let $z, s\in\ZZ$ satisfying $\ref{item1}$
and $\ref{item2}$, and let $C(z,s)/\Q$ be the curve defined in \eqref{eqn:C}. Let $q \in \Z$ be an odd prime of bad reduction. Then, its conductor exponent is given by the following table.    
 \begin{table}[H]
\begin{tabular}{|c||l|l|}
\hline
$\mathfrak{n}(C/\Q_q)$ & \text{Condition} & \text{Comments} \\ \hline
$\frac{r-1}{2}$ &  ${q} \nmid rs$ & \text{$F(x)$ is reducible over $\Q_q$} \\ \hline
$r-1$ & ${q} \neq {r}$, \, ${q} \mid s$ & \text{$F(x)$ is irreducible over $\Q_q$}
\\ \hline
$r-1$ & ${q} = {r}$, \, ${r} \nmid \Delta$,  $F(x)$ \text{ is reducible over } $\mathbb{Q}_r$ &  \\ \hline
$r$ & ${q} = {r}$, \, ${r} \nmid \Delta$, $F(x)$ \text{ is irreducible over } $\mathbb{Q}_r$ & \\ \hline
$2r-1$ & ${q} = {r}$, \, ${r} \mid s$, \, $r \mid \Delta$ & \text{$F(x)$ is irreducible over $\Q_r$}%
\\ \hline
$\frac{3r-1}{2}$ & ${q} = {r}$, \, ${r} \nmid s$, \, $v_{{r}}(\Delta) = 1$ & \text{$F(x)$ is irreducible over $\Q_r$} \\ \hline
$r$ & ${q} = {r}$, \, ${r} \nmid s$,  \, $v_{{r}}(\Delta) = 2$ & \text{$F(x)$ is irreducible over $\Q_r$} \\ \hline
$r-1$ & ${q} = {r}$, \, ${r} \nmid s$, \, $v_{r}(\Delta) \geq 3$ & \text{$F(x)$ is reducible over $\Q_r$} \\ \hline
\end{tabular}
{\small
\caption{Conductor exponent of $C(z,s):y^2 = F(x)$ over $\Q_q$, where $\Delta = s^2-4z^r$.}
    \label{table:conductorQ}}
\end{table}
\end{thm}\begin{proof} 
Let $q\neq r$. Since $q$ is assumed to be a prime of bad reduction, then $q\mid \Delta$. If $q\nmid s$, by Lemmas \ref{lemma:lematecnico} and  \ref{lemma:differenceroots} we have that $v_q(\gamma_k-\gamma_j)>0$ if and only if $j=r-k$, for any $1 \le j, k \le r-1$. Therefore, the reduction of $F(x)$ mod $q$ has at most double roots and, consequently, $C$ has semistable reduction over $q$ (see the discussion after Proposition 2.11 in \cite{BCDF}). Then, we can use \cite[Corollary 9.4]{DDMM} to compute the conductor exponent, which is the number of twins in the cluster picture. By case (1) of Theorem \ref{thm:clustersQ}, this is equal to ${(r-1)}/{2}$. 

When $q\mid s$, Theorem \ref{thm:ramification} implies that $q\nmid e_{L/\Q_q}$, therefore the extension is tamely ramified and the result follows directly from Proposition \ref{prop:conductorFred}. 

\vspace{5pt}

From now on, we will assume $q=r$. By Proposition \ref{prop:conductorFred}, it only remains to compute the wild component, which, by Theorem \ref{thm:ramification}, can be non-zero only when $F(x)$ is irreducible over $\Q_r$, and therefore $\deg(F(x)) = r$. In this situation, the second case of Theorem \ref{thm:formulascluster} applies, leading to compute $v_r(\Delta_{\Q_r(\gamma_0)/\Q_r})$ and $f_{\Q_r(\gamma_0)/\Q_r}$, where the latter is equal to 1, by Corollary~\ref{coro:totallyramified}. Finally, the value of $v_r(\Delta({\Q_r(\gamma_0)/\Q_r}))$ was computed for all cases under consideration in Lemma \ref{lemma:valuationdiscriminants}, and so the computation of the wild conductor is straightforward, giving that 
\[\mathfrak{n}_{\text{wild}}(C/\Q_r) = \delta - r + 1,\]
and so
\begin{equation}
\label{eqn:wildconductorcases}
\mathfrak{n}_{\text{wild}}(C/\Q_r) = \begin{cases}
1 & \text{ if } r\nmid \Delta,\\
r 
& \text{ if } r\mid \Delta, \ r\mid s,\\
\frac{r+1}{2} & \text{ if } v_r(\Delta) = 1, r \nmid s,\\
1 & \text{ if } v_r(\Delta) = 2, r \nmid s, \\
\end{cases}
\end{equation}
whence the expression for the conductor in the statement of the theorem follows. We recall that, if $r \nmid s$ and $v_r(\Delta) \ge 3$, the polynomial $F(x)$ is reducible by Lemma \ref{lemma:reducible} and therefore, its wild conductor is $0$. {Finally, the comments about the irreducibility of $F(x)$ follow directly from Lemma \ref{lemma:reducible}.}
\end{proof}

Let $\id{q}$ be an odd prime of $K=\Q(\omega)$ and let $K_\id{q}(\mathcal{R})$ be the splitting field of $F(x)$. Since $\omega$ lies in the splitting field of $F(x)$ (see Proposition \ref{prop:splittingfield} and Remark \ref{rmk:splittingfield}), we can consider $K_\id{q}(\mathcal{R})$ as an extension field of $K_\id{q}=\Q_q(\omega)$. Let $\id{r}$ be the unique prime in $K$ above $r$.

\begin{prop} \label{prop:tameoverK} Let $\id{q}$ be an odd prime of $K$ such that $\id{q}\mid rs$. The tame component is equal to  $\mathfrak{n}_\text{tame}(C/K_\id{q})=r-1$. In particular,  if $F(x)$ is reducible over $\Q_r$, then $\mathfrak{n}(C/K_\id{r})=r-1$.
\end{prop}

\begin{proof} 

We can mimic the proof of Proposition \ref{prop:conductorFred} by proving that $\#U/I_{K_\id{r}}=\#V/I_{K_\id{r}}$. Case (1) can be verified following exactly the same steps and for cases (2) and (3), 
one can note that the only odd clusters different from $\mathcal{R}$ are the singletons $\{\gamma_i\}$, which have $P(\gamma_i) = \mathcal{R}$. In addition, $\mathcal{R}$ is the only proper non-übereven cluster in all cases, so it suffices to compute $\xi_\mathcal{R}(\tilde{\lambda}_\mathcal{R})$ and $\xi_\mathcal{R}(d_\mathcal{R})$, using Corollary \ref{coro:clusteroverK}. For both cases, we get
\[\xi_\mathcal{R}(\tilde{\lambda}_\mathcal{R})=2>1=\xi_\mathcal{R}(d_\mathcal{R}),
\]
and so $U=V=\emptyset$. We now proceed to the two subcases (4.1) and (4.2).

\vspace{3pt}

\textbf{(4.1) Case} $\boldsymbol{{r \nmid \text{}} s, \ 1 \le v_r(\Delta)\le 2}$: 
As in the previous cases, the only odd clusters different from $\mathcal{R}$ are the singletons and the only proper non-übereven cluster is $\mathcal{R}$. Omitting details, we find that 
\[U = \begin{cases}
\{\{\gamma_i\} \mid 0 \le i \le r-1\} & \text{ if } v_{{r}}(\Delta) = 1 \text{ and } r \equiv 3 \pmod{8}, \\
\emptyset & \text{ otherwise,}
\end{cases}
\]
and 
\[V = \begin{cases}
\{\mathcal{R}\} & \text{ if } v_{{r}}(\Delta) = 1 \text{ and } r \equiv 3 \pmod{8}, \\
\emptyset & \text{ otherwise.}
\end{cases}
\]
 If $U, V \neq \emptyset$, Theorem \ref{thm:ramification} gives that the extension $K_{\id{r}}(\mathcal{R})/K_{\id{r}}$ has ramification degree
\[e_{K_{\id{r}}(\mathcal{R})/K_{\id{r}}} = 
\frac{e_{K_{\id{r}}(\mathcal{R})/\Q_{{r}}}}{e_{K_\id{r}/\Q_r}} = \frac{r(r-1)}{2}\cdot\frac{2}{r-1} = r,
\]
and consequently, the ramification index of the extension $K_{\id{r}}(\gamma_0)/K_{\id{r}}$ is $r$, and so the sizes of orbits of inertia is $r$. Thus, we have that $\#U/{I_{K_{\id{r}}}} = 1 = \#V/{I_{K_{\id{r}}}}$.

\vspace{5pt}
\textbf{(4.2) Case} $\boldsymbol{{r\nmid \text{}} s, \ v_r(\Delta)\ge3}$: Note that in this case $F(x)$ is reducible, by Lemma \ref{lemma:reducible}, and the cluster picture is given by case (5) of
Corollary \ref{coro:clusteroverK}. In particular, we let $\mathbf{t}_i$ denote the twin $\{\gamma_i, \gamma_{r-i}\}$, for $ 1 \le i \le {(r-1)}/{2}$. We remark that the only cluster containing $\mathbf{t}_i$ is $\mathcal{R}$, which has $d_\mathcal{R} = 1$. Consequently, it follows from \cite[Theorem 1.3(iv)]{Bisatt} that $[I_{K_\id{r}}:I_{\mathbf{t}_i}] = 1$,
irrespective of whether $\mathbf{t}_i$ is an orphan or not. By following an identical approach to the proof of case (6) of Theorem \ref{thm:conductor}, we show that 

\[U = \begin{cases}
    \{\{\gamma_i\} \mid 1 \le i \le  r-1\} & \text{ if } 2 \mid \frac{r-1}{2}v_r(\Delta), \\
    \emptyset & \text{ otherwise,}
\end{cases}
\]

and
\[V = \begin{cases}
    \{\mathbf{t}_i \mid 1 \le i \le \frac{r-1}{2}\} & \text{ if } 2 \mid \frac{r-1}{2}v_r(\Delta), \\
    \emptyset & \text{ otherwise.}
\end{cases}
\]
If $2 \nmid {(r-1)}v_r(\Delta)/2$, $U=V=\emptyset$. If $2 \mid {(r-1)}v_r(\Delta)/2$, Lemma \ref{lemma:indexram} gives that the ramification degree of the extension $K_\id{r}(\mathcal{R})/K_\id{r}$ is $2$. This fact, together with $[I_{K_{\id{r}}}:I_{\mathbf{t}_i}] = 1$, means that the inertia group permutes the two roots belonging to a twin, but that the twins themselves are stable under the action of inertia. Consequently, both $\#U/{I_{K_\id{r}}}$ and $\#V/{I_{K_\id{r}}}$ are equal to $ {(r-1)}/{2}$, finishing the proof of the proposition.
\end{proof}

\begin{thm}
\label{thm:conductoroverK}
Let $z, s\in\ZZ$ satisfying $\ref{item1}$
and $\ref{item2}$, and let $C(z,s)/K$ be the curve defined in \eqref{eqn:C}. Let $\id{q}$ be an odd prime of $K$ of bad reduction  above the rational prime $q$. Then, its conductor exponent is given by the following table.
    \begin{table}[H]
\begin{tabular}{|c||l|l|}
\hline
$\mathfrak{n}(C/K_\id{q})$ & \text{Condition} & \text{Comments} \\ \hline
$\frac{r-1}{2}$ &  ${q} \nmid rs$ & \text{$F(x)$ is reducible over $\Q_q$}  \\ \hline
$r-1$ & ${q} \neq {r}$, \, ${q} \mid s$ & \text{$F(x)$ is irreducible over $\Q_q$}
\\ \hline
$r-1$ & ${q} = {r}$, \, ${r} \nmid \Delta$, $F(x)$ \text{ is reducible over } $\mathbb{Q}_r$ &  \\ \hline
$\frac{3(r-1)}{2}$ & ${q} = {r}$, \, ${r} \nmid \Delta$, $F(x)$ \text{ is irreducible over } $\mathbb{Q}_r$ & \\ \hline
$\frac{(r-1)(r+2)}{2}$ & ${q} = {r}$, \, ${r} \mid s$, \, $r \mid \Delta$ & \text{$F(x)$ is irreducible over $\Q_r$}
\\ \hline
$\frac{(r-1)(r+5)}{4}$ & ${q} = {r}$, \, ${r} \nmid s$, \, $v_{{r}}(\Delta) = 1$ & \text{$F(x)$ is irreducible over $\Q_r$} \\ \hline
$\frac{3(r-1)}{2}$ & ${q} = {r}$, \, ${r} \nmid s$,  \, $v_{{r}}(\Delta) = 2$ & \text{$F(x)$ is irreducible over $\Q_r$} \\ \hline
$r-1$ & ${q} = {r}$, \, ${r} \nmid s$, \, $v_{r}(\Delta) \geq 3$ & \text{$F(x)$ is reducible over $\Q_r$} \\ \hline

\end{tabular}
{\small \caption{Conductor exponent of $C(z,s):y^2 = F(x)$ over $K_\id{q}$, where $\Delta = s^2-4z^r$.}
    \label{table:conductorK} }
\end{table}
\end{thm}
\begin{proof}
 When ${q}\nmid rs$, the proof mimics the one of Theorem \ref{thm:conductor}. 
 If ${q}\neq {r}$ and ${q}\mid s$ we have from Theorem \ref{thm:ramification} that the extension $K_\id{q}(\mathcal{R})/K_\id{q}$ is tamely ramified so the result follows from Proposition \ref{prop:tameoverK}.

By Proposition \ref{prop:tameoverK}, it only remains to compute the wild part, which is only non-zero when ${q}={r}$ and $F(x)$ is irreducible over $\Q_r$. Note that, since $K_{\id{r}} = \Q_r(\omega)$, the extension $K_\id{r}/\Q_r$ is totally ramified of degree ${(r-1)}/{2}$ and so it is tamely ramified with ramification index $e_{K_\id{r}/\Q_\id{r}} = {(r-1)}/{2}$. Therefore, the computation of $\mathfrak{n}_\text{wild}(C/K_\id{r})$ is straightforward from
\cite[Lemma 2.2]{CelineClusters} and the value of $\mathfrak{n}_\text{wild}(C/\Q_r)$ calculated in (\ref{eqn:wildconductorcases}). We have that $\mathfrak{n}_\text{wild}(C/K_\id{r})=0$ if $F(x)$ is reducible over $\Q_r$ and, when $F(x)$ is irreducible, we get
\begin{equation}\label{eq:wild-over-K}
\mathfrak{n}_\text{wild}(C/K_\id{r})=\begin{cases}
\frac{r-1}{2} & \text{ if } r\nmid \Delta,\\
\frac{r(r-1)}{2} 
& \text{ if } r\mid \Delta, \ r\mid s,\\
\frac{(r+1)(r-1)}{4} & \text{ if } v_r(\Delta) = 1, r \nmid s,\\
\frac{r-1}{2} & \text{ if } v_r(\Delta) = 2, r \nmid s. \\
\end{cases}
\end{equation}
\end{proof}

Recall that in step (3) of the modular method we have to compute some space of newforms with level given by the conductor of the curve. If the level is too large, this step could be unfeasible. For this reason, by a computational point of view, the following result may be useful in the future.

\begin{coro} Suppose that $v_r(\Delta)\ge 3$ and that $r \nmid s$. Consider the curve $C'/{K_\id{r}}$, which is the quadratic twist of $C$ by a uniformizer of $K_\id{r}$. Then, $C'$ is semistable and has conductor exponent ${(r-1)}/{2}$.
\end{coro}

\begin{proof} 
We begin by noting that the leading coefficient of the polynomial defining $C'$ now has valuation exactly equal to $1$. Since the cluster picture is invariant under quadratic twists, we have that 
\[\nu_\mathcal{R}(C') = 1 + r\cdot 1 = r+1 \in 2\Z.
\]
On the other hand, since $v_r(\Delta) \ge 3$, Lemma \ref{lemma:reducible} gives that the polynomial $F(x)$ is reducible over $\Q_r$, and then,  by Proposition \ref{prop:gamma0}, $\gamma_0 \in \Q_q$. Then, Lemma \ref{lemma:indexram} gives that $e_{K_\id{q}(\mathcal{R})/K_\id{q}}\le 2$. 

In addition, it can be shown that all the twins in the cluster picture are invariant under the action of inertia. We may then apply \cite[Definition 1.7 and Theorem 1.8]{DDMM} to conclude that $C'/K_{\id{r}}$ is semistable. Then, the computation of the conductor follows directly from \cite[Corollary 9.4]{DDMM}, since there are precisely $({r-1})/{2}$ even clusters, corresponding to the twins.
\end{proof}

\section{Diophantine applications}
\label{sec:applications}

In this section we elaborate on the discussion in the introduction and show how  
Theorem~\ref{thm:conductoroverK} can be used to recover and extend some of the previously existing results in the literature pertaining to the study of the solutions $(a,b,c)$ to the GFE \eqref{eq:GFE}, where $\gcd(Aa, Bb, Cc) = 1$. This assumption is standard in the literature and will simplify computations.

\vspace{5pt}

To understand where the curve $C(z,s)$ arises from and the role that it plays, we must go back to \cite{TTV}, where it is shown that the equation
\begin{align}\label{eq:hypOriginal}
    C(u)\colon y^2=(-1)^\frac{r-1}{2}xh(2-x^2)+u, \quad u\in\Q,
\end{align}
defines a family of hyperelliptic curves of genus ${(r-1)}/{2}$ whose Jacobians contain $K$ in their endomorphism algebra. Moreover, the endomorphisms of $J:= \text{Jac}(C(u))$ are defined over $K$, implying that $J$
becomes of $\GL_2$-type over $K$.

This turns out to be a useful piece of machinery to build the two-dimensional Galois representation needed when following the modular method (see e.g.\ \cite[Theorem A.8]{CelineClusters}, or \cite[p.\ 420]{Darmon} for a more detailed explanation about the construction).

\begin{lemma} \label{lemma:ultimo} Let $\delta \in \Q$. Then the curve
   $C(\delta^2 z, \delta^r s)$ is isomorphic to the quadratic twist of $C(z, s)$ by $\delta$.
\end{lemma}
\begin{proof}
If we twist $C(z,s)$ by $\delta$ we get the model 
\begin{equation}\label{eq:lemmaultimo}
\delta y^2 = (-z)^\frac{r-1}{2}xh(2-x^2)+s.
\end{equation}
Then, the result follows by multiplying both sides (\ref{eq:lemmaultimo}) by $\delta^r$ and making the change of variables $x\mapsto x/\delta$, $y\mapsto y/\delta^{(r+1)/2}$.
\end{proof}
\begin{lemma} \label{lemma:twist} 
Let $z, s\in\ZZ$. Suppose that either 
\begin{enumerate}
\item $z$ is a perfect square, or
\item $\Delta=s^2-4z^r$ is a perfect square.
\end{enumerate}
Then the Jacobian of $C(z,s)$ is of $\GL_2$-type over $K$.

\end{lemma}
\begin{proof}
Firstly, suppose that (1) holds, so $z=z_0^2$ for some $z_0\in\ZZ$. Let $u=s/z_0^r$. Then, by Lemma~\ref{lemma:ultimo} we have that $C(z,s)=C(z_0^2,z_0^ru)$ is the quadratic twist by $z_0$ of $C(1,u)=C(u)$. Hence, the result follows from the fact that $J$ is of $\GL_2$-type over $K$, by \cite[Theorem 1]{TTV},  since the Jacobian of any quadratic twist of $C(u)$ remains of $\GL_2$-type over $K$.

\vspace{5pt}

Secondly, suppose that (2) holds, so that $\Delta=d^2$, for some $d\in\ZZ$.  Let $t = (s/d+1)/2$ and $\alpha = \sqrt{z^r/\Delta}$. We note that 
\begin{align}\label{eq:CofGL2type} C' := C\Big(\frac{z^r}{\Delta}, \Big(\frac{z^r}{\Delta}\Big)^\frac{r-1}{2} \cdot \frac{s}{d}\Big) 
      = C\left(\alpha^2, \alpha^{r-1}(2t-1)\right),
\end{align}
 and therefore corresponds to the curve $C'_r(t)$ defined in \cite[Lemma 2.36]{BCDF}. By \cite[Theorem 2.38]{BCDF}, the Jacobian of $C'$ is of $\GL_2$-type over $K$. 
Finally, the result follows by noting that we can write $C'$ as 
\[C'=C\Big(\frac{z^{r-1}}{d^2}\cdot z,\frac{z^\frac{r(r-1)}{2}}{d^r}\cdot s\Big),\]
which, by Lemma \ref{lemma:ultimo}, is the quadratic twist of $C(z,s)$ by $z^{(r-1)/2}/d$.
\end{proof}

Let us assume that the Jacobian of $C(z,s)$ if of $\GL_2$-type over $K$. In this case it has associated a strictly compatible system of two-dimensional Galois representations (in the sense of \cite[Definition 4.10]{Bockle}) that we denote by $\{\rho_{J, \lambda}:\Gal(\bar{K}/K)\to\GL_2(K_\lambda)\}$ (see \cite[Section 11.10]{Shimura}), {where $\lambda$ is a place of $K$}. Let $\id{q}$ be a prime of $K$. It follows by \cite[$\S 8$]{Ulmer} that the conductor of $\rho_{J,\lambda}$ at $\id{q}$ is independent of $\lambda$, for all $\lambda$ satisfying $\id{q}\nmid\text{Norm}(\lambda)$.

When following the modular method to approach the GFE (\ref{eq:GFE}), one wants to consider a prime $\id{p}$ of $K$ above $p$, and then the  residual representation modulo $p$ of the $\id{p}$-member of this family, i.e., $\bar{\rho}_{J,\id{p}}:\Gal(\bar{K}/K)\to \GL_2(\F_p)$. This is the representation described in step (1) of the method.


In the subsequent subsections, we will compute  the conductor exponent of ${\rho}_{J, \id{p}}$ at all odd places. To compute the conductor of $\bar{\rho}_{J,\id{p}}$ some large image results are needed, but that is out of the scope of this article.

We shall make use of the following result, which is an extension of \cite[Proposition A.12]{CelineClusters}, since their proof applies to our curve. 

\begin{proposition}
    \label{prop:jacobian} Keep the previous notation, and let $\id{q}$ be a prime ideal of $K$. 
    Then, 
    \[\mathfrak{n}_\id{q}(\rho_{J, \id{p}}) = \frac{2}{r-1}\mathfrak{n}(C/K_\id{q}).\]
\end{proposition}
Since our results only apply to odd places, we shall refer to the \textit{odd conductor}, and write $\mathfrak{n}_{\text{odd}}(\rho_{J,\id{p}})$, to refer to the part of the conductor of $\rho_{J,\id{p}}$ which does not include the even places. 

\subsection{\texorpdfstring{Signature $(p,p,r)$}{Signature (p,p,r)}} Let $(a,b,c)$ be a non-trivial solution of (\ref{eq:GFE}) for signature $(p,p,r)$. In other words,  the equality $Aa^p+Bb^p=Cc^r$ holds. Recall that we further assume that $\gcd(Aa,Bb,Cc)=1$. Moreover, without loss of generality, we can assume that $C$ is free of $r$-th powers (otherwise it induces a solution of a GFE with same signature satisfying such a condition).

To achieve a curve compatible with the signature $(p,p,r)$, we set the values
\[z = C^2c^2 \quad \text{and} \quad s = 2C^{r-1}(Bb^p-Aa^p),
\]
so that the curve $C(z,s)$ from (\ref{eqn:C}) becomes
\[C^-_{r}(a,b): y^2 = (-1)^{\frac{r-1}{2}} (Cc)^{r-1}xh\left(2-\frac{x^2}{C^2c^2}\right) + 2C^{r-1}(Bb^p-Aa^p), \]
with Jacobian $J_r^-$ over $K$. Note that, for $A = B = C = 1$, we recover the curve $C_r^-(a,b)$ originally studied by Darmon in \cite{Darmon}.
In this case,
\[\Delta = s^2-4z^{r} = -16ABC^{2(r-1)}a^pb^p,
\]
and the discriminant of the curve $C_r^-(a,b)$ is, by Lemma \ref{lemma:discriminant},
\[\Delta_{C^-_{r}(a,b)}= 2^{4(r-1)}r^rC^{(r-1)^2}(ABa^pb^p)^\frac{r-1}{2}.
\]
In this case we have  $F(x)=(-1)^{\frac{r-1}{2}} (Cc)^{r-1}xh\left(2-\frac{x^2}{C^2c^2}\right) + 2C^{r-1}(Bb^p-Aa^p)$.

\begin{coro}

    Let $\id{r}$ be the unique prime ideal of $K$ above $r$, and define $\varepsilon_r$ with the expression
    \[\varepsilon_r = \begin{cases}
        2 & \text{ if } r \nmid ABCab \text{ and } F(x) \text{ is reducible over } \Q_r, \\
        3 & \text{ if } r \nmid ABCab \text{ and } F(x) \text{ is irreducible over } \Q_r, \\
        2 &\text{ if } r \mid ab, \\
  
       \frac{r+5}{2} & \text{ if } r \nmid ab \text{ and  } v_r(AB) = 1,\\

        3 & \text{ if } r \nmid ab \text{ and  } v_r(AB)=2, \\

        2 & \text{ if } r \nmid ab \text{ and  } v_r(AB) \ge 3, \\
        r+2 & \text{ if } r \mid C.       
    \end{cases}
    \]
    Then, the odd conductor of the representation $\rho_{J^-_{r}, \id{p}}$, is given by 
    \[\mathfrak{n}_{\text{odd}}(\rho_{J^-_r, \id{p}}) = \id{r}^{\varepsilon_r} \prod_{\substack{\id{q} \mid ABab, \\ \id{q} \nmid r}} \mathfrak{q} \prod_{\substack{\id{q} \mid C, \\ \id{q} \nmid r}} \mathfrak{q}^2.
    \]
\end{coro}

\begin{proof}
First, we notice that $z$ and $s$ satisfy conditions \ref{item1} and \ref{item2}. This is clear, since, given a prime $q$ dividing both $z$ and $s$, we have that $q\mid C$ and so \[v_q(z^r)=2rv_q(C)>2(r-1)v_q(C)=v(s^2).\] Condition \ref{item2} follows from the fact that $C$ is free of $r$-th powers.

Notice that $z$ is a square, so Lemma \ref{lemma:twist} applies. 
Then, the results follows by applying Theorem \ref{thm:conductoroverK}  and Proposition \ref{prop:jacobian}, along with the coprimality assumption.
\end{proof}
Note that when $A = B = C = 1$ and $r=5$ we get that $\varepsilon_r \in\{2,3\}$, so we recover the result in \cite[Theorem 7.9]{ChenKoutsianas}. For a more general prime $r\ge5$, the result matches the values obtained in \cite{CelineClusters}. 

\subsection{\texorpdfstring{Signature $(r,r,p)$}{Signature (r,r,p)}}
 Let $(a,b,c)$ be a non-trivial solution of (\ref{eq:GFE}) for signature $(r,r,p)$. In other words,  the equality $Aa^r+Bb^r=Cc^p$ holds. Recall that we further assume that $\gcd(Aa,Bb,Cc)=1$. Similarly as before, without loss of generality we can assume that $A$ and $B$ are free of $r$-th powers.

The first instance of a Frey hyperelliptic curve compatible with the signature $(r,r,p)$ is due to Kraus \cite{Kraus}, who considered the case $A=B=C=1$. In this subsection, we generalize his curve by allowing for arbitrary values of $A$, $B$ and $C$. For this, we set $s$ and $z$ to be 
\[z = {-ABab} \quad \text{and} \quad s =  {(AB)^\frac{r-1}{2}(Bb^r-Aa^r)},
\]
so that the curve $C(z,s)$ from (\ref{eqn:C}) becomes 
\begin{align*}
    C_{r,r}(a,b): y^2 = (ABab)^\frac{r-1}{2}xh\left(2+\frac{x^2}{ABab}\right) + (AB)^\frac{r-1}{2}(Bb^r-Aa^r),
\end{align*}
with Jacobian $J_{r,r}$ over $K$.
The value of $\Delta$ in this case turns out to be
\begin{equation}\label{eq:deltarrp}
\Delta = s^2-4z^{r} = {(AB)^{r-1}C^2c^{2p}},
\end{equation}
and, from Lemma \ref{lemma:discriminant}, the discriminant of $C_{r,r}(a,b)$ is 
\[\Delta_{C_{r,r}(a,b)} = (-1)^\frac{r-1}{2} 2^{2(r-1)} r^r (AB)^\frac{(r-1)^2}{2} (Cc^p)^{r-1}. \]
In this case we have  $F(x)= (ABab)^\frac{r-1}{2}xh\left(2+\frac{x^2}{ABab}\right) + (AB)^\frac{r-1}{2}(Bb^r-Aa^r)$.
\begin{corollary}

    Let $\id{r}$ be the unique prime ideal of $K$ above $r$, and define $\varepsilon_r$ with the expression
    \[\varepsilon_r = \begin{cases}
        2 & \text{ if } r \nmid ABCc \text{ and } F(x) \text{ is reducible over } \Q_r, \\
        3 & \text{ if } r \nmid ABCc \text{ and } F(x) \text{ is irreducible over } \Q_r, \\
        r+2 & \text{ if } r \mid AB, \\
        3 & \text{ if } r \nmid c \text{ and } v_r(C) = 1, \\
        2 & \text{ if } r \mid c \text{ or } v_r(C) \ge 2. \\
    \end{cases}
    \]
                   Then, the odd conductor of the representation $\rho_{J_{r,r}, \id{p}}$ is given by 
    \[\mathfrak{n}_{\text{odd}}(\rho_{J_{r,r}, \id{p}}) = \id{r}^{\varepsilon_r} \prod_{\substack{\id{q} \mid AB, \\ \id{q} \nmid r}} \mathfrak{q}^2 \prod_{\substack{\id{q} \mid Cc, \\ \id{q} \nmid r}} \mathfrak{q}.
    \]
\end{corollary}

\begin{proof}
First, we notice that $z$ and $s$ satisfy conditions \ref{item1} and \ref{item2}. This is clear, since, given a prime $q$ dividing both $z$ and $s$, we have that $q\mid AB$ and so \[v_q(z^r)=rv_q(AB)>(r-1)v_q(AB)=v(s^2).\] Condition \ref{item2} follows from the fact that $A$ and $B$ are free of $r$-th powers.

Notice that, from (\ref{eq:deltarrp}), $\Delta$ is a square, so Lemma \ref{lemma:twist} applies. 
Then, the results follows by applying Theorem \ref{thm:conductoroverK}  and Proposition \ref{prop:jacobian}, along with the coprimality assumption.
\end{proof}
In the case where $A = B = C = 1$, we find that $a-b$ is a rational root of $F(x)$, so $F(x)$ is reducible, therefore $\varepsilon_r = 2$. This recovers the result of \cite[Theorem 5.16]{BCDF}.
The result matches the values obtained in \cite{CelineClusters}. 

\vspace{7pt}

\begin{remark} \label{rem:martin}
 While this article was under the revision process, an independent work by Azon was posted on arXiv \cite{Martinpreprint}. Following a similar methodology, the author in \cite{Martinpreprint} obtains the same results for the conductor at odd places for signatures $(p,p,r)$ and $(r,r,p)$. For readers interested in a comparison of both articles, we provide below a brief correspondence between the two notations (see \cite[Table 1]{Martinpreprint}):
 \begin{table}[H]
\renewcommand{\arraystretch}{1.1}
\begin{tabular}{|c||c|}
\hline
Notation in \cite{Martinpreprint} & Notation in this paper \\
\hline
$\delta_\Q$ & $\sqrt{z}$ \\
\hline
$\delta_\Q^r(2-4s_0)$ & $s$ \\
\hline
$2^4\delta_\Q^{2r}s_0(s_0-1)$ & $\Delta=s^2-4z^r$\\
\hline
\end{tabular}
\caption{Dictionary with the notation in \cite{Martinpreprint}. \label{table:dictorionary}}
\end{table}

\end{remark}


\subsection{\texorpdfstring{Signature $(2,r,p)$}{Signature (2,r,p)}}
Let $(a,b,c)$ be a non-trivial solution of (\ref{eq:GFE}) for signature $(2,r,p)$. In other words,  the equality $Aa^2+Bb^r=Cc^p$ holds. Recall that we further assume that $\gcd(Aa,Bb,Cc)=1$. Similarly as before, without loss of generality we can assume that $B$ is free of $r$-th powers.

A Frey hyperelliptic curve for signature $(2,r,p)$ and coefficients $A=B=C=1$ was introduced in \cite[Section 1.4.4]{ChenKoutsianasSurvey}. Here we will extend such a construction for arbitrary values of $B$ and $C$, by showing that the new Frey curve fits within our framework.
To get a curve for this equation, we set the following values:
\[z={-Bb} \quad  \text{ and } \quad  s=2B^\frac{r-1}{2}a,\]
so that the curve $C(z,s)$ of (\ref{eqn:C}) becomes 
\[C_{2,r}(a,b) : y^2 = (Bb)^\frac{r-1}{2}xh\left(2+\frac{x^2}{Bb}\right)+2B^\frac{r-1}{2}a,\]
with Jacobian $J_{2,r}$ over $K$.
Note that for $B=C=1$ it matches the curve mentioned in \cite[Section 1.4.4]{ChenKoutsianasSurvey}. In this case, 
\[\Delta=s^2-4z^{r}=4B^{r-1}Cc^p,\]
and the discriminant of the curve $C_{2,r}(a,b)$ is, by Lemma \ref{lemma:discriminant},
\[\Delta_{C_{2,r}(a,b)}= 2^{3(r-1)}r^r(-B^{r-1}Cc^p)^\frac{r-1}{2}.
\]
Here we have $F(x)=(Bb)^\frac{r-1}{2}xh\left(2+\frac{x^2}{Bb}\right)+2B^\frac{r-1}{2}a$.

In this case neither $z$ nor $\Delta$ is a perfect square, therefore Lemma \ref{lemma:twist} cannot be applied. However, we can still compute the odd conductor of $C_{2,r}(a,b)$ over $K$ by using Theorem \ref{thm:conductoroverK}. Moreover, if we assume that the Jacobian of the of $C_{2,r}(a,b)$ is of $\GL_2$-type over $K$, we may use Proposition \ref{prop:jacobian} and get the following result.

\begin{corollary} Let $\id{r}$ be the unique prime ideal of $K$ above $r$, and define $\varepsilon_r$ with the expression
    \[\varepsilon_r = \begin{cases}
        2 & \text{ if } r \nmid BCc \text{ and } F(x) \text{ is reducible over } \Q_r, \\
        3 & \text{ if } r \nmid BCc \text{ and } F(x) \text{ is irreducible over } \Q_r, \\
        r+2 & \text{ if } r \mid B, \\
        \frac{r+5}{2} & \text{ if } r \nmid Bc \text{ and } v_r(C) = 1, \\
        3 & \text{ if } r \nmid Bc \text{ and } v_r(C) = 2, \\
        2 & \text{ if } r \nmid B \text{ and either } r \mid c \text{ or } v_r(C)\ge3.       
    \end{cases}
    \]
Then, the odd conductor of the representation $\rho_{J_{2,r}, \id{p}}$ is given by 
    \[\mathfrak{n}_{\text{odd}}(\rho_{J_{2,r}, \id{p}}) = \id{r}^{\varepsilon_r} \prod_{\substack{\id{q} \mid B, \\ \id{q} \nmid r}} \mathfrak{q}^2 \prod_{\substack{\id{q} \mid Cc, \\ \id{q} \nmid Br}} \mathfrak{q}.
    \]
\end{corollary}
\begin{proof} By the recent work of the second author with Chen \cite[Sections 2.3.3 and 4.4]{ChenVillagra}, $J_{2,r}$ is of $\GL_2$-type over $K$ and so we only have to prove that conditions \ref{item1} and \ref{item2} are satisfied. Given a prime $q$ dividing both $z$ and $s$, we have that $q\mid B$ and so \[v_q(z^r)=rv_q(B)>(r-1)v_q(B)=v(s^2).\] Condition \ref{item2} follows from the fact that $B$ is free of $r$-th powers.
\end{proof}

The following remark shows it is sufficient to prove non-existence of non-trivial primitive solutions for some GFE equation with signature $(2,r,p)$, in order to prove non-existence of non-trivial primitive solutions of the GFE with signature $(r,r,p)$. In fact, it is sufficient to consider only equations with $A = 1$.

\begin{remark} If $(a,b,c)$ is a solution to 
\begin{align}\label{eq:2rp}
    Ax^r+By^r=Cz^p, \quad \gcd(Ax,By,Cz) = 1,
\end{align}
 then $(Bb^r-Aa^r,ab,c^2)$ is a solution to 
\begin{equation} \label{eq:blda}
    x^2+4ABy^r=C^2z^p, \quad \gcd(x, ABy, Cz) = 1.
\end{equation}
 In other words, in order to approach equations of signature $(r,r,p)$, we may restrict to the study of GFE of signature $(2,r,p)$ with certain coefficients. Also, we note that the coefficient $4$ can be removed from $(\ref{eq:blda})$ if we restrict to a particular type of solutions, namely those with $Cc$ even. Indeed, if $(a,b,c)$ is a solution of (\ref{eq:2rp}) and $Cc$ is even, then $\left(({Bb^r-Aa^r})/{2},ab,c^2\right)$ is a solution to $x^2+ABy^r=C^2z^p/4$, if $2\mid C$, and $\left(({Bb^r-Aa^r})/{2},ab,c^2/4\right)$ is a solution to $x^2+ABy^r=C^2z^p$, if $2\mid c$.
 \end{remark}

 \begin{remark}

    If $A \neq 1$, we could use the following values of $z$ and $s$:
    \[z = -ABb, \quad \text{and} \quad s = 2A^\frac{r+1}{2}B^\frac{r-1}{2}a.\]
    In general, these values do not satisfy condition 
    \ref{item1}. We believe that a straightforward - but messy - adaptation of the techniques in this paper could allow to consider this case, but we have avoided it to allow for a clearer treatment overall.
    
 Some cases of this equation where $A \neq 1$ are currently being explored in work in progress of the two authors with Koutsianas.
 \end{remark}

\subsection{\texorpdfstring{Signature $(q,r,p)$}{Signature (q,r,p)}} 
\label{Sec:futurework}
In this subsection, and considering that we have already treated the cases $(p,p,r)$, $(r,r,p)$ and $(2,r,p)$, we restrict to the general case of signature $(q,r,p)$ where $q$, $r$ and $p$ are different odd prime numbers and $p$ is the only free variable. Recall that, while considering equation
\begin{align*}
    Ax^q+By^r=Cz^p,
\end{align*}
we study integer solutions satisfying $\gcd(aA,bB,cC)=1$. Also, we will assume, without loss of generality, that $A$ is free of $q$-th powers and $B$ is free of $r$-th powers.

In the recent beautiful article \cite{GolfieriPacetti}, Golfieri and Pacetti suggest a modification to step (1) in the application of the modular method.
Instead of using geometric objects as the source for 
Galois representations, the authors consider different objects, called hypergeometric motives (we will abbreviate them as HGM and we invite the reader to read \cite[Section 4]{hgm} for the definition of HGMs), in order to get a characteristic zero representation (see \cite[page 5]{GolfieriPacetti}). The construction is very explicit, and in Section 2.6 of loc.\ cit.\ they show that the following hypergeometric motives can be attached to the GFE equation \eqref{eq:GFE}:
\begin{align*}
   &M_{q,r}^+(t)=\HGM\left(\left(\frac{1}{r},-\frac{1}{r}\right),\left(\frac{1}{q},-\frac{1}{q}\right)\Big|t\right),\\ 
   &M_{q,r}^-(t)=\HGM\left(\left(\frac{1}{2r},-\frac{1}{2r}\right),\left(\frac{1}{q},-\frac{1}{q}\right)\Big|t\right). 
\end{align*}

These HGMs are defined over the totally real field $F=\Q(\zeta_q)^+\Q(\zeta_r)^+$. 
If $(a,b,c)$ is a solution to (\ref{eq:GFE}) satisfying $\gcd(Aa,Bb,Cc)=1$, we will need to specialize the HGMs at the value $t_0=-{Aa^q}/{Bb^r}$, so that  $M_{q,r}^\pm(t_0)$ have monodromy matrices at $\{0,\infty,1\}$ of order $\{q,\infty,r\}$, respectively.
\vspace{5pt}

Let $\{\rho^{\pm}_\lambda\}_\lambda$ be the compatible system of Galois representations associated to  $M_{q,r}^\pm(t_0)$. If $\id{p}$ is an ideal of $\cO_F$ above $p$, the conductor exponent of $\bar{\rho}^\pm_\id{p}$ at primes not dividing $qr$ can be computed from \cite[Propositions 2.11 and 2.13]{GolfieriPacetti}. All that remains is to compute the conductor exponent for primes dividing $qr$ and our contribution allows to compute the wild conductor exponent, as shown in the following proposition. First, let 
\[F(x) = (-B^2b^2)^\frac{r-1}{2}xh\Big(2-\frac{x^2}{B^2b^2}\Big)+2B^{r-1}(Bb^r+2Aa^q). \]

\begin{proposition} Keep the previous notation, and let $\id{r}$ be a prime of $\cO_F$ above $r$. Then the wild conductor exponent of $\bar{\rho}^\pm_\id{p}$ at $\id{r}$ equals $0$ if $F(x)$ is reducible. If $F(x)$ is irreducible, then it equals
\begin{align*}
\mathfrak{n}_{\text{wild},\id{r}}(\bar{\rho}^\pm_\id{p})=\begin{cases}
    0 & \text{if } r\nmid \Delta,\\
     \frac{r+1}{2} & \text{if } v_r(\Delta)=1,\\
    1 & \text{if }  v_r(\Delta)\in\{0,2\},\\
    r 
    & \text{if } r\mid B,\\    
    \end{cases}
\end{align*}
where $\Delta=16B^{2(r-1)}ACa^qc^p$.
\end{proposition}
\begin{proof}
Since $\{\rho_\lambda^\pm\}_\lambda$ is a compatible system of Galois representations, then 
\begin{align} \label{eq:wildconductor}
    \mathfrak{n}_{\text{wild},\id{r}}(\bar{\rho}^\pm_\id{p})=\mathfrak{n}_{\text{wild},\id{r}}(\bar{\rho}^\pm_\id{P}), \quad  \text{ for all } \id{P}\nmid r.
\end{align}
This is a well-known fact, and we refer the reader to \cite[Proposition 3.1.42]{Wise} for a proof.

 Let $\id{q}$ and $\id{p}_2$ be primes of $\cO_F$ above $q$ and $2$, respectively.  Since both HGMs $M_{q,r}^\pm$ are congruent modulo $\id{p}_2$, by \cite[Theorem 1.10]{GolfieriPacetti} (originally stated in \cite[Theorem 8.3]{originalHGM}), so are their Galois representations. Therefore, by (\ref{eq:wildconductor}), it is sufficient to  
show the result for $M_{q,r}^+$. Again, by \cite[Theorem 1.10]{GolfieriPacetti} (or \cite[Theorem 8.3]{originalHGM}), $M_{q,r}^+(t_0)$ is congruent modulo $\id{q}$ to the motive
\[\HGM\left(\left(\frac{1}{r},-\frac{1}{r}\right),\left(1,-1\right) \Big| t_0\right),\] 
 which turns out to match the curve $C_r^+(t_0)$ from \cite{Darmon} (as shown in \cite[Corollary A.5]{GolfieriPacetti}). In other words, the representation $\bar{\rho}^+_\id{q}$ is isomorphic to $\bar{\rho}_{J_r^+(t_0),\id{q}}|_{G_F}$ , where $J^+_r(t)$ is the Jacobian of $C^+_r(t)$ and $G_F$ is the absolute Galois group of $F=\Q(\zeta_r)^+\Q(\zeta_q)^+$. Let $\id{q}_r$ be a prime in $K=\Q(\zeta_r)$ above $r$. Since $\id{q}_r$ is unramified in $F/K$, then $\mathfrak{n}_{\text{wild}}(C_r^+/F_\id{r})=\mathfrak{n}_{\text{wild}}(C_r^+/K_{\id{q}_r})$, by \cite[Lemma 2.2]{CelineClusters}. That is, it is sufficient to consider the wild conductor over $K$.

Let us consider the following additional hyperelliptic curve introduced in \cite{Darmon}:
\[C_r^-(t): y^2=(-1)^\frac{r-1}{2}xh(2-x^2)+2-4t.\] 
By \cite[Remark 2.10]{CelineClusters}, $\mathfrak{n}_{\text{wild}}(C_r^+(t_0))=\mathfrak{n}_{\text{wild}}(C_r^-(t_0))$. 

\vspace{5pt}

Set $z_0=B^2b^2$ and $s_0=2B^{r-1}(Bb^r+2Aa^q)$. First, let us show that $z_0$ and $s_0$ satisfy conditions \ref{item1} and \ref{item2}. If $\ell$ is an odd prime dividing both $z_0$ and $s_0$, then $\ell\mid B$. Since $\gcd(aA,bB,cC)=1$, $\ell\nmid Bb^r +2Aa^q$ and so
\[v_\ell(z_0^r)\ge2rv_\ell(B)>2(r-1)v_\ell(B)=v_\ell(s_0^2).\]
On the other hand, since we are assuming that $B$ is free of $r$-th powers, then $r\nmid v_\ell(s_0)$.

\vspace{5pt}

Recall that we want to specialize the HGM at $t_0=-Aa^q/Bb^r$. Since $C_r^-(t_0)=C(1,2-4t_0)$, by Lemma \ref{lemma:ultimo} we can twist by $Bb$ to get the model
\[C(z_0,s_0): y^2 = (-B^2b^2)^\frac{r-1}{2}xh\Big(2-\frac{x^2}{B^2b^2}\Big)+2B^{r-1}(Bb^r+2Aa^q).\]
By Lemma \ref{lemma:discriminant}, its discriminant equals
$\Delta_{C(z_0,s_0)}=(-1)^{\frac{r-1}{2}} 2^{2(r-1)} r^r \Delta^{\frac{r-1}{2}}$, with \[\Delta=s_0^2-4z_0^r=16B^{2(r-1)}ACa^qc^p.\]
 Specializing the computation of the wild conductor in (\ref{eq:wild-over-K}) at these values, together with Proposition \ref{prop:jacobian} (we can apply it since $z_0$ is a perfect square and hence Lemma \ref{lemma:twist} gives that $C(z_0,s_0)$ is of $\GL_2$-type) allows us to prove the result by taking $\id{P} = \id{q}$ in \eqref{eq:wildconductor}.
\end{proof}

\begin{remark}
Note that, since $p$ is the free variable and $q$ and $r$ are fixed, the conductor exponent at $q$ follows simply by interchanging $q$ and $r$, $A$ and $B$ and $a$ and $b$ in the previous proposition.
\end{remark}


\bibliographystyle{alpha}
	\bibliography{biblio}

\end{document}